\documentclass{compositio}

\usepackage{amsmath,amsfonts, amsthm, amssymb}
\usepackage{array}
\usepackage[all,textures]{xy}

\theoremstyle{plain}
\newtheorem{theorem}{Theorem}[section]
\newtheorem{proposition}[theorem]{Proposition}
\newtheorem{lemma}[theorem]{Lemma}
\newtheorem{corollary}[theorem]{Corollary}

\theoremstyle{definition}
\newtheorem{definition}[theorem]{Definition}

\theoremstyle{remark}

\newtheorem{remark}[theorem]{Remark}
\newtheorem{example}[theorem]{Example}
\newtheorem*{question}{Question}

\renewcommand{\contentsname}{Contents\\{\footnotesize\normalfont(A table
of contents should normally not be included)}}

\newcommand{\embr}{\mathrm{embr}}

\renewcommand{\lim}{\mathrm{lim}}

\newcommand{\arit}{\mathrm{ar}}
\newcommand{\degr}{\mathrm{deg}}

\newcommand{\tw}{\mathsf{tw}}

\newcommand{\twilnil}{\mathsf{tw_{ilnil}}}
\newcommand{\Tw}{\mathsf{Tw}}

\newcommand{\Twfree}{\mathsf{Tw_{free}}}

\newcommand{\Twinil}{\mathsf{Tw_{ilnil}}}
\newcommand{\Twilnil}{\mathsf{Tw_{ilnil}}}

\newcommand{\Free}{\mathsf{Free}}

\newcommand{\Com}{\mathsf{Com}}
\newcommand{\PCom}{\mathsf{PCom}}

\newcommand{\dott}{\mathsf{dot}}

\newcommand{\Ext}{\mathrm{Ext}}

\newcommand{\Hom}{\mathrm{Hom}}

\newcommand{\Def}{\mathrm{Def}}
\newcommand{\pDef}{\mathrm{pDef}}

\newcommand{\Coder}{\mathrm{Coder}}

\newcommand{\Sk}{\mathrm{Sk}}

\newcommand{\op}{^{\mathrm{op}}}
\newcommand{\Ob}{\mathrm{Ob}}
\newcommand{\Z}{\mathbb{Z}}

\newcommand{\N}{\mathbb{N}}

\newcommand{\AAA}{\mathfrak{a}}
\newcommand{\BBB}{\mathfrak{b}}

\newcommand{\ZZZ}{\mathfrak{Z}}

\newcommand{\III}{\mathfrak{i}}

\newcommand{\TTT}{\mathfrak{t}}

\newcommand{\CC}{\mathbf{C}}

\newcommand{\Mod}{\ensuremath{\mathsf{Mod}} }

\newcommand{\ddef}{\ensuremath{\mathsf{def}} }

\newcommand{\dgcat}{\ensuremath{\mathsf{dgcat}}}
\newcommand{\cat}{\ensuremath{\mathsf{cat}}}

\newcommand{\hodgcat}{\ensuremath{\mathsf{hodgcat}} }

\newcommand{\Ind}{\ensuremath{\mathsf{Ind}}}

\newcommand{\Inj}{\ensuremath{\mathsf{Inj}}}

\newcommand{\lra}{\longrightarrow}

\newcommand{\aaa}{\ensuremath{\mathcal{A}}}
\newcommand{\bbb}{\ensuremath{\mathcal{B}}}

\newcommand{\ttt}{\ensuremath{\mathcal{T}}}

\begin{document}

\title{Hochschild cohomology, the characteristic morphism and derived deformations}
\author{Wendy Lowen}
\email{wlowen@vub.ac.be}
\address{Departement of Mathematics\\ Faculty of Sciences\\ Vrije Universiteit Brussel\\ Pleinlaan
2\\1050 Brussel\\ Belgium}
\thanks{The author is a postdoctoral fellow with FWO/CNRS. She acknowledges the hospitality of the Institut de Math\'ematiques de Jussieu (IMJ) and of the Institut des Hautes Etudes Scientifiques (IHES) during her postdoctoral fellowship with CNRS.}

\begin{abstract}
A notion of  Hochschild cohomology $HH^{\ast}(\aaa)$ of an abelian category $\aaa$ was defined by Lowen and Van den Bergh (2005) and they showed the existence of a characteristic morphism $\chi$ from the Hochschild cohomology of $\aaa$ into the graded centre $\ZZZ^{\ast}(D^b(\aaa))$ of the bounded derived category of $\aaa$. An element $c \in HH^2(\aaa)$ corresponds to a first order deformation $\aaa_c$ of $\aaa$ (Lowen and Van den Bergh, 2006). The problem of deforming an object $M \in D^b(\aaa)$ to $D^b(\aaa_c)$ was treated by Lowen (2005).
In this paper we show that the element $\chi(c)_M \in \Ext_{\aaa}^2(M,M)$ is precisely the obstruction to deforming $M$ to $D^b(\aaa_c)$. Hence this paper provides a missing link between the above works. Finally we discuss some implications of these facts in the direction of a ``derived deformation theory''.
\end{abstract}

\maketitle

\section{Introduction}

Let $k$ be a commutative ring. It is well known that for a $k$-algebra $A$, there is a characteristic morphism $\chi_{A}$ of graded commutative algebras from the Hochschild cohomology of $A$  to the graded centre of the derived category $D(A)$.
If $k$ is a field, this morphism is determined by the maps, for $M \in D(A)$,
$$M \otimes^{L}_A -: HH^{\ast}_k(A) \cong \Ext^{\ast}_{A^{\op} \otimes_k A}(A,A) \lra \Ext^{\ast}_A(M,M)$$
The characteristic morphism plays an important r\^ole for example in the theory of support varieties (\cite{buchweitzavramov}, \cite{erdmann}, \cite{solberg}).
Characteristic morphisms were generalized to various situations where a good notion of Hochschild cohomology is at hand.
Recently, Buchweitz and Flenner defined and studied Hochschild cohomology for \emph{morphisms} of schemes or analytic spaces, and proved the existence of a characteristic morphism in this context (\cite{buchweitzflenner3}).
In \cite{keller6}, Keller defined the Hochschild cohomology of an \emph{exact} category as the Hochschild cohomology of a certain dg quotient. For an \emph{abelian} category $\aaa$, this is precisely the Hochschild cohomology of a ``dg enhancement'' of the bounded derived category $D^b(\aaa)$. Consequently, the projection on the zero part of the Hochschild complex (see \S \ref{parproj}) is itself a natural dg enhancement of a characteristic morphism
$$\chi_{\aaa}: HH^{\ast}_{\mathrm{ex}}(\aaa) \lra \ZZZ^{\ast}(D^b(\aaa))$$
where the right hand side denotes the graded centre of $D^b(\aaa)$ (see \S \ref{parcentre}).
Explicitely, $\chi_{\aaa}$ maps a Hochschild $n$-cocycle $c$ to a collection of elements $\chi_{\aaa}(c)_M \in \Ext^{\ast}_{\aaa}(M,M)$ for $M \in D^b(\aaa)$. 
The main purpose of this paper is to give an interpretation of $\chi_{\aaa}(c)_M$ in terms of deformation theory. 

In \cite{lowenvandenbergh1}, a deformation theory of abelian categories was developed. Its relation with Hochschild cohomology goes through an alternative definition of the latter given by the authors in \cite{lowenvandenbergh2}, and shown in the same paper to be equivalent to Keller's definition. Let us consider, from now on, an abelian category $\aaa$ with enough injectives, and let us assume that $k$ is a field. Then $\Inj(\aaa)$ is a $k$-linear category and we put
\begin{equation}\label{firsteq}
HH^{\ast}_{\mathrm{ab}}(\aaa) = HH^{\ast}(\Inj(\aaa))
\end{equation}
The main advantage of $\Inj(\aaa)$ is that, considering it as a ring with several objects, its deformation theory is entirely understood in the sense of Gerstenhaber's deformation theory of algebras (\cite{gerstenhaber}). It is shown in \cite{lowenvandenbergh1} that the \emph{abelian} deformation theory of $\aaa$ is equivalent to the \emph{linear} deformation theory of $\Inj(\aaa)$, justifying \eqref{firsteq}.
An abelian deformation $\bbb$ of $\aaa$ gives rise to a morphism 
\begin{equation}\label{secondeq}
D^b(\bbb) \lra D^b(\aaa)
\end{equation}
and an obstruction theory for deforming objects $M \in D^b(\aaa)$ to $D^b(\bbb)$, which is the subject of \cite{lowen2}. The main theorem of the current paper (see also Theorem \ref{maintheorem}) states

\begin{theorem}\label{mainintro}
Consider $c \in HH^2_{\mathrm{ex}}(\aaa)$ and let $\aaa_c$ be the corresponding (first order) deformation of $\aaa$. For $M \in D^b(\aaa)$, the element $\chi_{\aaa}(c)_M \in \Ext_{\aaa}^2(M,M)$ is the obstruction against deforming $M$ to an object of $D^b(\aaa_c)$.
\end{theorem}

Hence, the characteristic morphism $\chi_{\aaa}$ is a natural ingredient in a theory describing the simultaneous deformations of an abelian category together with (families or diagrams of) objects in the abelian (or derived) category. The details of this theory remain to be worked out.

In \cite{lowen2}, \eqref{secondeq} is expressed in terms of complexes of injectives in $\aaa$ and $\bbb$, and the obstructions are expressed in terms of the element $c \in HH^2_{\mathrm{ab}}(\aaa) = HH^2(\Inj(\aaa))$ corresponding to the abelian deformation. Essentially, our approach for proving the above theorem is tightening the relation between $HH^{\ast}_{\mathrm{ab}}(\aaa)$ and $HH^{\ast}_{\mathrm{ex}}(\aaa)$.

For a differential graded category $\AAA$, let $\CC(\AAA)$ denote its Hochschild complex (\cite{keller6}). Let $D_{\mathrm{dg}}^b(\aaa)$ be a dg model of $D^b(\aaa)$ constructed using complexes of injectives. The natural inclusion $\Inj(\aaa) \subset D_{\mathrm{dg}}(\aaa)$ induces a projection morphism
\begin{equation}\label{thirdeq}
\CC(D_{\mathrm{dg}}(\aaa)) \lra \CC(\Inj(\aaa))
\end{equation}
which is proven in \cite{lowenvandenbergh2} to be a quasi-isomorphism of $B_{\infty}$-algebras. The $B_{\infty}$-structure of the Hochschild complexes captures all the operations relevant to deformation theory, like the cup product and the Gerstenhaber bracket, but also the more primitive brace operations (see \S \ref{parhochainf}). 
In \S \ref{parsection}, we explicitely construct a $B_{\infty}$-section
$$\embr_{\delta}: \CC(\Inj(\aaa)) \lra \CC(D_{\mathrm{dg}}(\aaa))$$
of \eqref{thirdeq} (Theorem \ref{anotherembr2}). In the notation, $\delta$ is the element in $\CC^1(D_{\mathrm{dg}}(\aaa))$ determined by the differentials of the complexes of injectives, and $\embr$, short for ``embrace'', refers to the brace operations. More concretely, for $c \in \CC^n(\Inj(\aaa))$, we have
$$\embr_{\delta}(c) = \sum_{m = 0}^nc\{ \delta^{\otimes m} \}$$
After introducing the characteristic morphism in \S \ref{parchar}, we use the morphism $\embr_{\delta}$ to prove Theorem \ref{mainintro} in \S \ref{parcharobs}. 

The morphism  $\embr_{\delta}$ also throws some light on the following question, which is part of a research project in progress:
\begin{question}
Given an abelian deformation $\bbb$ of an abelian category $\aaa$, in which sense can we interpret $D^b(\bbb)$ as a ``derived'' deformation of $D^b(\aaa)$?
\end{question}

More precisely, the morphism $\embr_{\delta}$ gives us a recipe to turn a linear deformation of $\Inj(\aaa)$ (and hence an abelian deformation of $\aaa$) into a deformation of $D_{\mathrm{dg}}^b(\aaa)$... as a \emph{cdg} category! Here cdg, as opposed to dg, means that apart from compositions $m$ and differentials $d$, the category has ``curvature elements'' correcting the fact that $d^2 \neq 0$. This ``small'' alteration has serious consequences, ruining for example the classical cohomology theory.

In Theorem \ref{abelian} we show that the cdg deformation of $D_{\mathrm{dg}}^b(\aaa)$ contains a maximal partial dg deformation which, at least morally, is precisely $D^b_{\mathrm{dg}}(\bbb)$ (see also Remarks \ref{abelianremark} and \ref{boundedremark}). The part of  $D_{\mathrm{dg}}^b(\aaa)$ that gets dg deformed in this way is spanned by the ``zero locus'' of the characteristic element 
$$(M \mapsto \chi_{\aaa}(c)_M) \in \prod_{\Ob(D^b(\aaa))}\Ext^2_{\aaa}(M,M)$$
Hence, an object $M \in D^b(\aaa)$ contributes to the dg deformation of $D^b_{\mathrm{dg}}(\aaa)$ if and only if it deforms, in the sense of \cite{lowen2}, to an object of $D^b(\bbb)$.

\section{$A_{[0,\infty[}$-categories}

$A_{\infty}$-algebras and categories are by now widely used as algebraic models for triangulated categories (see \cite{lyubashenko2}, \cite{hamiltonlazarev}, \cite{keller7}, \cite{kontsevichsoibelman} and the references therein). Although the generalization to the $A_{[0,\infty[}$-setting causes serious new issues, a large part of the theory can still be developed ``in the $A_{\infty}$-spirit''. In this section we try to give a brief, reasonably self contained account of the facts we need. For more detailed accounts we refer the reader to \cite{getzlerjones2}, \cite{lazarev} for $A_{\infty}$-algebras, to \cite{lefevre}, \cite{lyubashenko} for $A_{\infty}$-categories and to \cite{nicolas} for $A_{[0,\infty[}$-algebras.

\subsection{A word on signs and shifts}\label{parsign}

Let $k$ be a commutative ring. All the algebraic constructions in this paper take place in and around the category $G(k)$ of $\Z$-graded $k$-modules. For $M, N \in G(k)$, we have the familiar tensor product 
$$(M \otimes N)^n = \oplus_{i \in \Z}M^i \otimes N^{n-i}$$
and internal hom 
$$[M,N]^p = \prod_{i \in \Z}\Hom(M^i, M^{i+p})$$
over $k$. For $m \in M^i$, the \emph{degree} of $m$ is $|m| = i$. We adopt the \emph{Koszul sign convention}, i.e. $G(k)$ is endowed with the well known closed tensor structure with ``super'' commutativity isomorphisms
\begin{equation}\label{eqnsupercom}
M \otimes N \lra N \otimes M: m \otimes n \longmapsto (-1)^{|m| |n|}n \otimes m
\end{equation}
and the standard associativity and identity isomorphisms. The closed structure is determined by the evaluation morphism
$$[M,N] \otimes M \lra N: (f,m) \longmapsto f(m)$$
Furthermore, we make a choice of shift functors on $G(k)$. For $i \in \Z$, let $\Sigma^ik \in G(k)$ be the object whose only nonzero component is $(\Sigma^ik)^{-i} = k$.
The \emph{shift functors} are the functors
$$\Sigma^i = \Sigma^ik \otimes - : G(k) \lra G(k): M \longmapsto \Sigma^iM = \Sigma^ik \otimes M$$
For $m \in M$, we put $\sigma^im = 1 \otimes m \in \Sigma^iM$.
All the canonical isomorphisms (and in particular the signs) in this paper are derived from the above conventions. The most general canonical isomorphisms we will use are of the form, for  $M_1, \dots , M_n, M \in G(k)$: 
\begin{equation}\label{eqnmix}
 \varphi : \Sigma^{i - i_1 - \dots - i_n}[M_1 \otimes \dots \otimes M_n, M] \lra [\Sigma^{i_1}M_1 \otimes \dots \otimes \Sigma^{i_n}M_n, \Sigma^{i} M]
 \end{equation}
 defined by $\varphi(\sigma^{i - i_1 - \dots - i_n}\phi)(\sigma^{i_1}m_1, \dots ,\sigma^{i_n}m_n) = (-1)^{\alpha}\sigma^{i}\phi(m_1, \dots, m_n)$
 where $$\alpha = {(i_1 + \dots + i_n)|\phi| + i_2|m_1| + \dots + i_n(|m_1| + \dots + |m_{n-1}|)}$$

\subsection{The Hochschild object of a (graded) quiver}\label{parquiver}

In this section and the next one we will introduce the Hochschild complex of an $A_{[0,\infty[}$-category (see also \cite{getzlerjones2}, \cite{lazarev}) in two steps. Our purpose is to distiguish between the part of the structure that comes from the $A_{[0,\infty[}$-structure (next section) and the part that does not (this section). This will be useful later on when we will transport $A_{[0,\infty[}$-structures.

Let $k$ be a commutative ring. A \emph{graded $k$-quiver} is a quiver enriched in the category $G(k)$. More precisely, a graded $k$-quiver $\AAA$ consists of a set of objects $\Ob(\AAA)$ and for $A, A' \in \Ob(\AAA)$, a graded object $\AAA(A,A') \in G(k)$. Since we will only use \emph{graded $k$}-quivers in this paper, we will systematically call them simply \emph{quivers}. The category of quivers with a fixed set of objects admits a tensor product
$$\AAA \otimes \BBB(A, A') = \oplus_{A''} \AAA(A'', A') \otimes \BBB(A, A')$$ and an internal hom
$$[\AAA, \BBB](A,A') = [\AAA(A,A'), \BBB(A,A')]$$
We put $[\AAA, \BBB] = \prod_{A,A'}[\AAA,\BBB](A,A')$. A \emph{morphism of degree $p$ from $\AAA$ to $\BBB$} is by definition an element of $[\AAA, \BBB]^p$. 
The \emph{tensor cocategory} $T(\AAA)$ of a quiver $\AAA$ is the quiver
$$T(\AAA) = \oplus_{n \geq 0} \AAA^{\otimes n}$$
equiped with the comultiplication $\Delta: T(\AAA) \lra T(\AAA) \otimes T(\AAA)$ which separates tensors.
There are natural notions of morphisms and of coderivations between cocategories and there is a $G(k)$-isomorphism
$$[T(\AAA), \AAA] \cong \Coder(T(\AAA), T(\AAA))$$
The object $[T(\AAA), \AAA]$ is naturally a brace algebra. We recall the definition.

\begin{definition}\label{bracedef} (see also \cite{gerstenhabervoronov})
For $V \in G(k)$, the structure of \emph{brace algebra} on $V$ consists in the datum of (degree zero) operations
$$V^{\otimes n+1} \lra V: (x,x_1, \dots x_n) \longmapsto x\{ x_1, \dots x_n\}$$
satisfying the relation $x\{x_1, \dots x_m\}\{y_1, \dots y_n\} =$
$$\sum (-1)^{\alpha}x\{y_1, \dots, x_1\{ y_{i_1}, \dots\} ,y_{j_1} \dots, x_m\{y_{i_m} , \dots\}, y_{j_m}, \dots y_n\}$$
where $\alpha = \sum_{k = 1}^m |x_k|\sum_{l =1}^{i_k - 1}|y_l|$. The \emph{associated Lie bracket} of a brace algebra is
$$\langle x,y \rangle = x\{y\} - (-1)^{|x||y|} y\{x\}$$ 
A \emph{brace algebra morphism} (between two brace algebras) is a graded morphism preserving all the individual brace operations.
\end{definition}

\begin{proposition}
Let $V$ be a brace algebra. The tensor coalgebra $T(V)$ naturally becomes a (graded) bialgebra with the associative multiplication $M: T(V) \otimes T(V) \lra T(V)$ defined by the compositions
$$M_{k,l}: V^{\otimes k} \otimes V^{\otimes l} \lra T(V) \otimes T(V) \lra T(V) \lra V$$
with $$M_{1,l}(x; x_1, \dots x_l) = x\{x_1, \dots, x_l\}$$
and all other components equal to zero. The unit for the multiplication is $1 \in k = V^{\otimes 0}$.
\end{proposition}
\begin{proof} This is standard (see \cite{getzlerjones}). A coalgebra morphism $M$ is uniquely determined by the components $M_{k,l}$ and the brace algebra axioms translate into the associativity of $M$.
\end{proof}

Put $$[T(\AAA), \AAA]_n = [\AAA^{\otimes n}, \AAA] = \prod_{A_0, \dots, A_n \in \AAA} [\AAA(A_{n-1},A_n) \otimes \dots \otimes \AAA(A_0, A_1), \AAA(A_0, A_n)]$$
The brace algebra structure on $[T(\AAA), \AAA] =\prod_{n\geq 0} [T(\AAA), \AAA]_n$is given by the operations
$$[T(\AAA),\AAA)]_n \otimes [T(\AAA),\AAA)]_{n_1} \otimes \dots \otimes [T(\AAA),\AAA)]_{n_k} \lra [T(\AAA),\AAA)]_{n-k + n_1 + \dots + n_k}$$
with
$$\phi\{\phi_1, \dots \phi_n\} = \sum \phi(1 \otimes \dots \otimes \phi_1 \otimes 1 \otimes \dots \otimes \phi_n \otimes 1 \otimes \dots \otimes 1)$$ 
The associated Lie bracket corresponds to the commutator of coderivations.
We put $B\AAA = T(\Sigma \AAA)$ and $\CC_{br}(\AAA) = [T(\Sigma\AAA), \Sigma\AAA] = [B\AAA, \Sigma \AAA]$.
Summarizing, the quiver $\AAA$ yields towering layers of (graded) algebraic structure:
\begin{enumerate}
\item[(0)] the \emph{quiver} $\AAA$, i.e. the graded objects $\AAA(A,A')$;
\item[(1)] the \emph{cocategory} $B\AAA = T(\Sigma \AAA)$;
\item[(2)] the \emph{brace algebra} $\CC_{br}(\AAA) = [B\AAA, \Sigma\AAA] \cong \Coder(B\AAA, B\AAA)$ which is in particular a \emph{Lie algebra} and
\item[(0')] the associated \emph{Hochschild object} $\CC(\AAA) = \Sigma^{-1}\CC_{br}(\AAA)$;
\item[(1')] the \emph{bialgebra} $T(\CC_{br}(\AAA)) = B\CC(\AAA)$;
\end{enumerate}
There is a natural inclusion
\begin{equation}\label{eqincl}
\CC_{br}(\AAA) \lra T(\CC_{br}(\AAA) = B\CC(\AAA)
\end{equation} of (2) into (1').

\subsection{The Hochschild complex of an $A_{[0,\infty[}$-category}\label{parhochainf}

\begin{definition} (\cite{getzlerjones2})
Let $\AAA$ be a quiver. An \emph{$A_{[0,\infty[}$-structure} on $\AAA$ is an element 
$b \in \CC_{br}^1(\AAA)$ satisfying
\begin{equation}\label{eqbsquare}
b\{b\} = 0
\end{equation}
The morphisms $$b_n: \Sigma \AAA^{\otimes n} \lra \Sigma \AAA$$
defining $b$ are sometimes called \emph{(Taylor) coefficients} of $b$.
The couple $(\AAA, b)$ is called an \emph{$A_{[0,\infty]}$-category}. If $b_0 = 0$, it is called an \emph{$A_{\infty}$-category}. If $b_n = 0$ for $n \geq 3$, it is called a \emph{cdg category}. If  $b_0 = 0$ and $b_n = 0$ for $n \geq 3$, it is called a \emph{dg category}.
\end{definition}

\begin{remark}\label{remainfty}
Consider $b \in \CC_{br}^1(\AAA)$.
\begin{enumerate}
\item The equation (\ref{eqbsquare}) can be written out completely in terms of the coefficients $b_n$ of $b$ (\cite{lefevre}, \cite{nicolas}).
\item If we consider $b$ as a coderivation inside $[B\AAA, B\AAA]^1$, then (\ref{eqbsquare}) is equivalent to $$b^2 = b \circ b = 0$$
\item If we consider $b$ as an element of the bialgebra $B\CC(\AAA)$ through (\ref{eqincl}), then (\ref{eqbsquare}) is equivalent to
\begin{equation}
b^2 = M(b,b) = 0
\end{equation}
where $M$ is the multiplication of $B\CC(\AAA)$.
\end{enumerate}
\end{remark}

The easiest morphisms to consider between $A_{[0,\infty[}$-categories are those with a fixed set of objects. To capture more general morphisms, one could follow the approach of \cite{lefevre} for $A_{\infty}$-categories.

\begin{definition} (\cite{getzlerjones2})
Consider $A_{[0,\infty[}$-categories $(\AAA,b)$ and $(\AAA',b')$ with $\Ob(\AAA) = \Ob(\AAA')$. A \emph{(fixed object) morphisms of $A_{[0,\infty[}$-categories} is a (fixed object) morphism of differential graded cocategories $f: B\AAA \lra B\AAA'$ (i.e. a morphism of quivers preserving the comultiplication and the differential). It is determined by morphisms
$$f_n: (\Sigma \AAA)^{\otimes n} \lra \Sigma \AAA'$$
for $n \geq 0$.
\end{definition}

An $A_{[0,\infty[}$-structure on $\AAA$ introduces a load of additional algebraic structure on the tower of \S \ref{parquiver}.
The \emph{Hochschild differential} on $\CC_{br}(\AAA)$ associated to $b$ is given by
$$d = \langle b,-\rangle \in [\CC_{br}(\AAA), \CC_{br}(\AAA)]^1: \phi \lra \langle b,\phi \rangle$$
and makes $\CC_{br}(\AAA)$ into a dg Lie algebra. The complex $\Sigma^{-1}\CC_{br}(\AAA)$ is (isomorphic to) the classical Hochschild complex of $\AAA$.
Similarly, considering $b \in B\CC(\AAA)^1$, we define a differential 
$$D = [b,-]_M \in [B\CC(\AAA),B\CC(\AAA)]^1: \phi \longmapsto [b,\phi]_M$$
where $[-,-]_M$ denotes the commutator of the multiplication $M$ determined by the brace operations. As $D$ is a coderivation, it defines an $A_{[0,\infty[}$-structure on $\CC(\AAA)$. Let us examine the coefficients
$$D_n: \Sigma \CC(\AAA)^{\otimes n} \lra \Sigma \CC(\AAA)$$
of this $A_{[0,\infty[}$-structure.
By definition,
$D_n(\phi_1, \dots, \phi_n) = M_{1,n}(b; \phi_1, \dots, \phi_n) - M_{n,1}(\phi_1, \dots, \phi_n; b)$, so for $n =1$ we have $$D_1(\phi) = \langle b, \phi \rangle$$ whereas for $n > 1$ we have
$$D_n(\phi_1, \dots, \phi_n) = b\{\phi_1, \dots, \phi_n\}$$
The differential $D$ makes $B\CC(\AAA)$ into a dg bialgebra. By definition, this makes $\CC(\AAA)$ into a \emph{$B_{\infty}$-algebra} \cite{getzlerjones}.
Summarizing, we obtain the following tower:
\begin{enumerate}
\item[(0)] the \emph{$A_{[0,\infty[}$-category} $\AAA$;
\item[(1)] the \emph{dg cocategory} $B\AAA = T(\Sigma\AAA)$;
\item[(2)] the \emph{$B_{\infty}$-algebra} $\CC_{br}(\AAA) = [B\AAA, \Sigma\AAA] \cong \Coder(B\AAA, B\AAA)$ which is in particular a \emph{dg Lie algebra} and
\item[(0')] the associated \emph{Hochschild complex} $\CC(\AAA) = \Sigma^{-1}\CC_{br}(\AAA)$;
\item[(1')] the \emph{dg bialgebra} $T(\CC_{br}(\AAA)) = B\CC(\AAA)$;
\end{enumerate}

By a \emph{$B_{\infty}$-morphism} (between $B_{\infty}$-algebras $B_1$ and $B_2$) we will always mean a graded morphism (super)commuting with all the individual operations on $B_1$ and $B_2$. A $B_{\infty}$-morphism is a brace algebra morphism, and a very particular case of a morphism of $A_{\infty}$-algebras.

\subsection{Limited functoriality}\label{limfunct}
The following tautological proposition will be used later on to transfer Hochschild cochains:
\begin{proposition}\label{transfer}
Consider quivers $\AAA$ and $\BBB$ and a brace algebra morphism $\Psi: \CC_{br}(\AAA) \lra \CC_{br}(\BBB)$. If $b$ is an $A_{[0,\infty[}$-structure on $\AAA$, then $\Psi(b)$ is an $A_{[0,\infty[}$-structure on $\BBB$ and
$$\Psi: \CC_{br}(\AAA, b) \lra \CC_{br}(\BBB, \Psi(b))$$ is a $B_{\infty}$-morphism.
\end{proposition}

\begin{proof}
For $\phi \in T(\CC_{br}(\AAA, b))$, we are to show that $T(\Psi)([b,\phi]) = [\Psi(b), T(\Psi)(\phi)]$, which immediately follows from the fact that $\Psi$ preserves the brace multiplication.
\end{proof}

Let $\BBB \subset \AAA$ be the inclusion of a full subquiver, i.e. $\Ob(\BBB) \subset \Ob(\AAA)$ and $\BBB(B,B') = \AAA(B,B')$. Using the induced $B\BBB \lra B\AAA$, there is a canonical restriction brace algebra morphism
$$\pi_{\BBB}: \CC_{br}(\AAA) \lra \CC_{br}(\BBB)$$
If $(\AAA,b)$ is an $A_{[0,\infty[}$-category, $\BBB$ can be endowed with the induced $A_{[0,\infty[}$-structure $\pi_{\BBB}(b)$ and $\pi_{\BBB}$ becomes a $B_{\infty}$-morphism (see Proposition \ref{transfer}). The $A_{[0,\infty[}$-category $(\BBB, \pi_{\BBB}(b))$ is called a \emph{full $A_{[0,\infty[}$-subcategory} of $\AAA$.
In particular, for every object $A \in \AAA$, $(\AAA(A,A), \pi_A(b))$ is an $A_{[0,\infty[}$-algebra.

\subsection{Projection on the zero part}\label{parproj}
Let $\AAA$ be an $A_{[0,\infty[}$-category. By the \emph{zero part} of $\CC_{br}(\AAA)$ we mean
$$\CC_{br}(\AAA)_0 = [T(\Sigma \AAA), \Sigma \AAA]_0 = \prod_{A \in \AAA} \Sigma \AAA(A,A)$$
Consider the graded morphisms $$\pi_0: \CC_{br}(\AAA) \lra \CC_{br}(\AAA)_0$$
and $\sigma_0: \CC_{br}(\AAA)_0 \lra \CC_{br}(\AAA)$. The morphism
$$b_1 \in [\Sigma \AAA, \Sigma \AAA]^1 = \prod_{A, A' \in \AAA}[\Sigma \AAA(A,A'), \Sigma \AAA(A,A')]^1$$
determines degree 1 morphisms $(b_1)_A: \Sigma \AAA(A,A) \lra \Sigma \AAA(A,A)$ and a product morphism
$$b_1^{\Delta}: \CC_{br}(\AAA)_0 \lra \CC_{br}(\AAA)_0$$
Let $d: \CC_{br}(\AAA) \lra \CC_{br}(\AAA)$ be the Hochschild differential.

\begin{proposition}\label{projzero}
We have $$b_1^{\Delta} = \pi_0 d \sigma_0$$
In particular, if $\AAA$ is an $A_{\infty}$-category,  $\pi_0: (\CC_{br}(\AAA),d) \lra (\CC_{br}(\AAA)_0, b_1^{\Delta})$ is a morphism of differential graded objects.
\end{proposition}

\begin{proof}
For an element $x \in \Sigma \AAA(A,A)$, we have $\pi_A(d(x)) = \langle b,x \rangle = b_1\{x\} = b_1(x)$.
\end{proof}

\subsection{From $\Sigma \AAA$ to $\AAA$}

The Hochschild complex of $\AAA$ and the $B_{\infty}$-structure on $\Sigma \CC(\AAA)$ are often expressed in terms of $\AAA$ rather than $\Sigma\AAA$. This can be done using the canonical isomorphisms
\begin{equation}\label{eqnbm}
\xymatrix{{\Sigma^{1-n}[\AAA(A_{n-1},A_n) \otimes \dots \otimes \AAA(A_0,A_1), \AAA(A_0,A_n)]} \ar[d] \\
{[\Sigma\AAA(A_{n-1},A_n) \otimes \dots \otimes \Sigma\AAA(A_0,A_1), \Sigma\AAA(A_0,A_n)]}}
\end{equation}
determined by the conventions of \S \ref{parsign}, thus introducing a lot of signs.
We define the bigraded object $\CC(\AAA)$ by
$$\CC^{i,n}(\AAA) = \prod_{A_0, \dots A_n} [\AAA(A_{n-1},A_n) \otimes \dots \otimes \AAA(A_0,A_1), \AAA(A_0,A_n)]^{i}$$
An element $\phi \in \CC^{i,n}$ has \emph{degree} $|\phi| = i$, \emph{arity} $\arit(\phi) = n$ and \emph{Hochschild degree} $\degr(\phi) = i + n$. We put $\CC^p(\AAA) = \prod_{i + n = p}\CC^{i,n}(\AAA)$.
The $B_{\infty}$-stucture of $\CC_{br}(\AAA)$ is translated in terms of operations on $\CC(\AAA)$ through (\ref{eqnbm}). The complex $\CC(\AAA)$ will also be called the Hochschild complex of $\AAA$ and its elements are called Hochschild cochains. For a Hochschild cochain
$\phi \in \CC^{i,n}(\AAA)$, the corresponding element of $\CC_{br}(\AAA)$ has
$$\sigma^{1-n}(\phi)(\sigma f_n, \dots , \sigma f_1) = (-1)^{ni + (n-1)|f_n| + \dots + |f_2|}\sigma\phi(f_n, \dots , f_1)$$
This identification is different from other ones, used for example in \cite{getzlerjones}, \cite{lefevre}, \cite{nicolas}. Nevertheless, it allows us to recover many standard constructions (up to minor modifications).
For example, the operation
$$\dott: \CC_{br}(\AAA)_n \otimes \CC_{br}(\AAA)_m \lra \CC_{br}(\AAA)_{n + m - 1}: (x,y) \longmapsto x\{y\}$$
gives rise to the classical ``dot product''
$$\bullet: \CC^{i,n}(\AAA) \otimes \CC^{j,m}(\AAA) \lra \CC^{i + j, n + m - 1}$$
on $\CC(\AAA)$ given by
\begin{equation}\label{eqnexpr}
\phi \bullet \psi = \sum_{k=0}^{n-1} (-1)^{\epsilon}\phi(1^{\otimes n-k-1} \otimes \psi \otimes 1^{\otimes k})
\end{equation}
where $$\epsilon = (\deg(\phi) + k + 1)(\arit(\psi)+1)$$

In the sequel, when no confusion arises, we will not distinguish in notation between the operations on $\CC_{br}(\AAA)$ and the induced operations on $\CC(\AAA)$. In particular, the brace operations will always be denoted using the symbols $\{$ and $\}$. An $A_{[0,\infty[}$-stucture $b \in \CC^1_{br}(\AAA)$ on $\AAA$ will often be translated into an element $\mu \in \CC^2(\AAA)$, which will also be called an $A_{[0,\infty[}$-stucture on $\AAA$. Similarly, we will speak about brace algebra and $B_{\infty}$-morphisms between Hochschild complexes $\CC(\AAA)$, $\CC(\AAA')$.

One proves:

\begin{lemma}\label{lembr}
Consider $\phi \in \CC^{i,n}(\AAA)$ and $\delta \in \CC^{j,0}(\AAA)$. 
We have $$\phi\{ \delta^{\otimes n} \} = (-1)^{n(i + ((n-1)/2)j)}\phi(\delta, \dots, \delta)$$
\end{lemma}

\subsection{The Hochschild complex of a cdg category}
By the previous section a cdg category is a graded quiver $\AAA$ together with 
\begin{enumerate}
\item 
\emph{compositions}
$\mu_2 = m \in  \prod_{A_0,A_1,A_2} [\AAA(A_1, A_2) \otimes \AAA(A_0, A_1),\AAA(A_0, A_2)]^0$
\item \emph{differentials} $\mu_1 = d \in \prod_{A_0,A_1}[\AAA(A_0, A_1), \AAA(A_0, A_1)]^1$
\item \emph{curvature elements} $\mu_0 = c \in \prod_A \AAA(A,A)^2$
\end{enumerate}
satisfying the identities:
\begin{enumerate}
\item $d(c) = 0$
\item $d^2 = -m(c \otimes 1 - 1 \otimes c)$
\item $dm = m(d \otimes 1 + 1 \otimes d)$
\item $m(m \otimes 1) = m(1 \otimes m)$
\end{enumerate}

\begin{remark}
Note that ii) differs from the conventional $d^2 = m(c \otimes 1 - 1 \otimes c)$ (\cite{getzlerjones2}, \cite{nicolas}). However, it suffices to change $c$ into $-c$ to recover the other definition.
\end{remark}

\begin{example}\label{pcomexample}
Let $\AAA$ be a linear category. An example of a cdg category is the category $\PCom(\AAA)$ of precomplexes of $\AAA$-objects. A \emph{precomplex of $\AAA$-objects} is a $\Z$-graded $\AAA$-object $C$ (with $C^i \in \AAA$) together with a $\Z$-graded $\AAA$-morphism $\delta_C: C \lra C$ of degree 1. We have $\PCom(\AAA)(C,D)^n = \prod_i\AAA(C^i, D^{i+n})$ and, for $f: D \lra E$ and $g: C \lra D$:
\begin{enumerate}
\item $m(f,g)_i = (fg)_i =  f_{i + |g|}g_i$
\item $d(f) = \delta_E f - (-1)^{|f|} f \delta_D$
\item $c_C = - \delta_C^2$
\end{enumerate}
Inside $\PCom(\AAA)$, we have the usual dg category $\Com(\AAA)$ of complexes $C$ of $\AAA$-objects, for which $c_C = \delta_C^2 = 0$. We will use the notations $\Com^+(\AAA)$ and $\Com^-(\AAA)$ for the respective categories of bounded below and bounded above complexes. 
\end{example}

As an example of the passage from $\Sigma\AAA$ to $\AAA$, let us use (\ref{eqnexpr}) to compute the Hochschild differential on $\CC(\AAA)$ for a cdg category $\AAA$. The differential on $\CC_{br}(\AAA)$ is given by
$$\langle \sigma c + d + \sigma^{-1} m, -\rangle$$
Consider $\phi \in \CC^{i,n}(\AAA)$. By definition 
$$\langle \sigma c + d + \sigma^{-1} m , \sigma^{1-n} \phi \rangle =
\dott(\sigma c + d + \sigma^{-1} m, \sigma^{1-n}\phi) - (-1)^{1-n +i}\dott( \sigma^{1-n}\phi, \sigma c + d + \sigma^{-1} m)$$
The corresponding three terms in terms of $\CC(\AAA)$ are:
\begin{enumerate}
\item $[c, \phi] = c \bullet \phi - (-1)^{\deg(\phi) + 1} \phi \bullet c$
which equals
$$\sum_{k=0}^{n-1}(-1)^{k+1}\phi(1^{\otimes{n-k-1}} \otimes c \otimes 1^{\otimes k})$$
\item $[d, \phi] = d \bullet \phi - (-1)^{\deg(\phi) + 1} \phi \bullet d$
which equals $$(-1)^{\arit{\phi}+1}(d\phi - (-1)^{|\phi|} \sum_{k=0}^{n-1}\phi(1^{\otimes n-k-1} \otimes d \otimes 1^{\otimes k}))$$
\item $[m, \phi] = m \bullet \phi - (-1)^{\deg(\phi) +1}\phi \bullet m$
which equals
$$m(\phi \otimes 1) + \sum_{k=0}^{n-1}(-1)^{k+1}\phi(1^{\otimes n-k-1} \otimes m \otimes 1^{\otimes k}) + (-1)^{n+1}m(1 \otimes \phi)$$
\end{enumerate}
If we look at the bigraded object $\CC^{i,n}$ with $i$ being the ``vertical'' grading and $n$ being the ``horizontal'' grading, then $d_h = [m,-]$ defines a \emph{horizontal contribution} whereas $d_v = [d,-]$ defines a \emph{vertical contribution} to the Hochschild differential $d$. Clearly, up to a factor $(-1)^{n+1}$, the horizontal contribution generalizes the classical Hochschild differential for an associative algebra. If we look at the ``$n$-th column'' graded object 
$$\CC^{\ast, n} =  \prod_{A_0, \dots A_n} [\AAA(A_{n-1},A_n) \otimes \dots \otimes \AAA(A_0,A_1), \AAA(A_0,A_n)]$$
then the vertical contribution on $\CC^{\ast, n}$ is $(-1)^{n+1}$ times the canonical map induced from $d$.
Compared to the dg case, we have a new \emph{curved contribution} $d_c = [c,-]$ which goes ``two steps up and one step back''. The curved contribution is zero on the zero part $\CC^{\ast,0}$. In the Hochschild complex of an arbitrary $A_{[0,\infty[}$-category, there are additional contributions going ``$n$ steps down and $n + 1$ steps ahead'' for $n \geq 1$.

\section{A $B_{\infty}$-section to twisted objects}\label{parsection}

Let $\AAA$ be a quiver. As explained in \S\ref{limfunct}, an inclusion $\AAA \subset \AAA'$ of $\AAA$ as a subquiver of some $\AAA'$ induces a morphism of brace algebras $\pi: \CC(\AAA') \lra \CC(\AAA)$. This section is devoted to the construction of certain quivers $\AAA' = \Tw(\AAA)$ of ``twisted objects over $\AAA$'' for which $\pi$ has a certain brace algebra section $\embr_{\delta}$. The morphism $\embr_{\delta}$ will be used in \S \ref{partrans} to transport $A_{[0,\infty[}$-structures from $\AAA$ to $\Tw(\AAA)$. Quivers of twisted complexes encompass the classical twisted complexes over a dg category (\cite{bondalkapranov}, \cite{drinfeld}, \cite{keller}), but also the ``infinite'' quivers of semifree dg modules (\cite{drinfeld}) as well as quivers of (pre)complexes over a linear category. The morphism $\embr_{\delta}$ is such that in those examples, it induces the correct $A_{[0,\infty[}$-structures on these quivers, thus defining a $B_{\infty}$-section of $\pi$. It will be used in \S \ref{parchar} to define the \emph{characteristic dg morphism} of a linear category $\AAA$, which allows us to prove Theorem \ref{maintheorem} and hence Theorem \ref{mainintro}. This chapter is related to ideas in \cite{fukayaooo}, \cite{fukaya}, \cite{lefevre}.

\subsection{Some quivers over $\AAA$}
Let $\AAA$ be a quiver. In this section we define the quiver $\Twfree(\AAA)$ of formal coproducts of shifts of $\AAA$-objects twisted by a morphism of degree $1$. First we define the quiver $\Free(\AAA)$. An object of $\Free(\AAA)$ is a formal expression $M = \oplus_{i \in I}\Sigma^{m_i}A_i$ with $I$ an arbitrary index set,  $A_i \in \AAA$ and $m_i \in \Z$. For another $N =\oplus_{j \in J}\Sigma^{n_i}B_i \in \Free(\AAA)$, the graded object $\Free(\AAA)(M,N)$ is by definition
$$\Free(\AAA)(M,N) = \prod_i \oplus_{j}\Sigma^{n_j - m_i}\AAA(A_i,B_j)$$
An element $f \in \Free(\AAA)(M,N)$ can be represented by a matrix $f = (f_{ji})$, where $f_{ji}$ represents the element $\sigma^{n_j - m_i}f_{ji}$.  

\begin{definition}
For $M, N$ as above, consider a morphism $f \in \Free(\AAA)(M,N)$. For a subset $S \subset I$, let $\Phi_f(S) \subset J$ be defined by $$\Phi_f(S) = \{ j \in J \, | \, \exists\, i \in S \, \, f_{ji} \neq 0\}$$ 
We say that $f \in \Free(M,M)$ is \emph{intrinsically locally nilpotent (iln)} if for every $i \in I$ there exists $n \in \N$ with  $\Phi_f^n(\{i\}) = \varnothing$. 
\end{definition}

\begin{proposition}\label{bracemap}
The canonical isomorphisms 
\begin{equation}\label{equcaniso}
\xymatrix{{[\AAA(A_{n-1},A_n) \otimes \dots \otimes \AAA(A_0,A_1), \AAA(A_0,A_n)]} \ar[d] \\
{[\Sigma^{i_n-i_{n-1}}\AAA(A_{n-1},A_n) \otimes \dots \otimes \Sigma^{i_1-i_0}\AAA(A_0,A_1), \Sigma^{i_n-i_0}\AAA(A_0,A_n)]}}
\end{equation}
for $A_k \in \AAA$, $n_k \in \Z$
define a morphism of brace algebras
\begin{equation}\label{equ1}
\CC(\AAA) \lra \CC(\Free(\AAA)): \phi \longmapsto \phi
\end{equation}
with
\begin{equation}\label{equexpr}
\phi(f_n, \dots, f_1)_{ji} = \sum_{k_{n-1}, \dots, k_1}(-1)^{\epsilon}\phi((f_n)_{jk_{n-1}}, (f_{n-1})_{k_{n-1}k_{n-2}}, \dots, (f_2)_{k_2k_1}, (f_1)_{k_1i})
\end{equation}
\end{proposition}

\begin{lemma}\label{univ}
Consider $\phi \in \CC(\AAA)$ and $(f_n, \dots , f_1) \in \Free(\AAA)(M_{n-1},M_n) \otimes \dots \otimes \Free(\AAA)(M_0,M_1)$. Write $M_0 = \oplus_{i \in I}\Sigma^{\alpha_i}A_i$ and consider $S \subset I$. There is an inclusion $$\Phi_{\phi(f_n, \dots, f_1)}(S) \subset \Phi_{f_n}(\Phi_{f_{n-1}}(\dots \Phi_{f_1}(S)))$$
\end{lemma}

\begin{proof}
Suppose $j$ is not contained in the right hand side. Then for every sequence $j = k_n, \dots, k_1, k_0 = i$ with $i \in S$ one of the entries $(f_p)_{k_p k_{p-1}}$ is zero. But then, looking at the expression (\ref{equexpr}), clearly $\phi(f_n, \dots, f_1)_{ji} = 0$ whence $j$ is not contained in the left hand side.
\end{proof}

Next we define the quiver $\Twfree(\AAA)$. An object of $\Twfree(\AAA)$ is a couple $(M,\delta_M)$ with $M \in \Free(\AAA)$ and $$\delta_M \in \Free(\AAA)(M,M)^1$$ 
For $(M, \delta_M), (N, \delta_N) \in \Twfree(\AAA)$, $\Twfree(\AAA)((M,\delta_M), (N, \delta_N)) = \Free(\AAA)(M, N)$. Consequently, the $\delta_M$ determine an element
$$\delta \in \CC^1(\Twfree(\AAA))$$
The isomorphisms (\ref{equcaniso}) also define a morphism of brace algebras
\begin{equation}
\CC(\AAA) \lra \CC(\Twfree(\AAA)): \phi \longmapsto \phi
\end{equation}
which is a section of the canonical projection morphism $\CC(\Twfree(\AAA)) \lra \CC(\AAA)$. In the next section we show that for certain $\AAA \subset \Tw \subset \Twfree$, 
\begin{equation}
\pi: \CC(\Tw(\AAA)) \lra \CC(\AAA) 
\end{equation} has another section depending on $\delta$, which can be used to transport $A_{[0,\infty[}$-structures.

\begin{definition}\label{deflocnilptw}
\emph{A quiver of locally nilpotent twisted objects over $\AAA$} is by definition a quiver $\Tw(\AAA)$ with $\AAA \subset \Tw(\AAA) \subset \Twfree(\AAA)$ such that for every $\phi \in \CC(\AAA)$, for every $$(f_n, \dots f_1) \in {\Tw(\AAA)(M_{n-1},M_n) \otimes \dots \otimes \Tw(\AAA)(M_0,M_1)}$$ with $M_0 = \oplus_{i \in I}\Sigma^{\alpha_i}A_i$, and for every $i \in I$ there exists $m_0 \in \N$ such that for all $m \geq m_0$, $\Phi_g(\{i\}) = \varnothing$ for $$g = \phi_{m+n}\{\delta^{\otimes m}\}(f_n, \dots f_1)$$
\end{definition}

\begin{example}\label{exdegzero}
If $\AAA$ is concentrated in degree zero, then $\Twfree(\AAA)$ is a quiver of locally nilpotent twisted objects over $\AAA$. Indeed, for $\phi \in \CC(\AAA)$, there is only a single $m$ for which the component $\phi_m$ is different from zero.
\end{example}

\begin{proposition}
Let $\Twinil(\AAA) \subset \Twfree(\AAA)$ be the quiver with as objects the $(M, \delta_M)$ for which $\delta_M \in \Free(\AAA)(M,M)$ is intrinsically locally nilpotent. Then $\Twinil(\AAA)$ is a quiver of locally nilpotent twisted objects over $\AAA$.
\end{proposition}

\begin{proof}
Consider $\phi$, $(f_n, \dots f_1)$ and $i$ as in Definition \ref{deflocnilptw} and put $\delta_i = \delta_{M_i}$. For $m \in \N$, consider $g_m = \phi_{m+n}\{\delta^{\otimes m}\}(f_n, \dots, f_1)$. This $g_m$ is a sum of expressions 
$$g_{m_n, \dots, m_0} = \phi_{m+n}(\delta_n^{\otimes m_n}, f_n, \delta_{n-1}^{\otimes m_{n-1}}, \dots, \delta_1^{\otimes m_1}, f_1, \delta_0^{\otimes m_0})$$
with $m_n + \dots + m_0 = m$.
For $\Phi_{g_{m_n, \dots, m_0}}(\{i\})$ to be empty, it suffices by Lemma \ref{univ} that
\begin{equation}\label{necsuf}
\Phi^{m_n}_{\delta_n}(\Phi_{f_n}(\dots(\Phi_{f_1}(\Phi^{m_0}_{\delta_0}(\{i\}))))) = \varnothing
\end{equation}
We recursively define numbers $p_l$ and finite sets $S_l$ for $l = 0, \dots, n$ in the following manner. Put $S_0 = \{i\}$. Once $S_{l}$ is defined, $p_{l}$ is such that $\Phi^{p_{l}}_{\delta_{l}}(S_{l}) = \varnothing$ (such a $p_{l}$ exists since $\delta_{l}$ is iln) and $S_{l+1} = \cup_{p \in \N}\Phi_{f_{l+1}}\Phi^{p}_{\delta_{l}}(S_{l})$. By the pigeonhole principle, if $m \geq p_n + \dots + p_0$, every $g_{m_n, \dots, m_0}$ with $m_n + \dots + m_0 = m$ has at least one $m_l \geq p_l$, and consequently (\ref{necsuf}) holds true. Hence in Definition \ref{deflocnilptw}, it suffices to take $m_0 = p_n + \dots + p_0$.
\end{proof}

\subsection{A word on topology}\label{partopo} Although not strictly necessary, it will be convenient to use a bit of topology to understand and reformulate definition \ref{deflocnilptw}. The language of this section will be used in the proof of Proposition \ref{bracesect}. All the topologies we consider will turn the underlying $k$-modules into topological $k$-modules, so in particular we can speak about completions.
Put $\CC = \CC_{br}(\Tw(\AAA))$ for some arbitrary full subcategory $\Tw(\AAA) \subset \Twfree(\AAA)$. To manipulate certain elements of $B_{\prod}\CC = \prod_{n \geq 0}(\Sigma\CC)^{\otimes n}$ that are not in $B\CC$, it will be convenient to consider a certain completion $\hat{B}\CC$ of $B\CC$. As a first step we endow $B\CC$ with a complete Hausdorff ``pointwise'' topology $\ttt_0$. To do so we suppose that $\AAA$ is naturally a complete Hausdorff topological $k$-quiver, i.e. the $\AAA(A,A')$ are complete Hausdorff topological $k$-modules (if there is no natural topology, the $\AAA(A,A')$ are endowed with the discrete topology). 

Now consider the algebra multiplication
$$M: B\CC \otimes B\CC \lra B\CC$$ defined by the brace operations. We suppose that $M$ preserves Cauchy nets with respect to $\ttt_0$. For every $\phi \in \CC_{br}(\AAA) \subset B\CC$ we consider the map
$$M_{\phi} = M(\phi,-): B\CC \lra (B\CC, \ttt_0)$$
Next we endow $B\CC$ with the ``weak topology'' $\ttt \subset \ttt_0$ which is by definition the initial topology for the collection $(M_{\phi})_{\phi}$, and we let $\hat{B}\CC$ denote the completion of $B\CC$ with respect to $\ttt$. The $M_{\phi}$ have natural continuous extensions
\begin{equation}\label{eqcont1}
\hat{M}_{\phi}: \hat{B}\CC \lra B\CC
\end{equation}

\begin{lemma}
For $\psi \in B\CC$, the map $M^{\psi} = M(-,\psi): B\CC \lra B\CC$ preserves Cauchy nets with respect to $\ttt$. Consequently, there is a natural continuous extension
\begin{equation}\label{eqcont2}
\hat{M}^{\psi}: \hat{B}\CC \lra \hat{B}\CC
\end{equation}
\end{lemma}
\begin{proof}
Suppose we have a $\ttt$-Cauchy net  $(x_{\alpha})_{\alpha}$ in $B\CC$. We have to show that $M(\phi, M(x_{\alpha}, \psi))$ is $\ttt_0$-Cauchy for every $\phi \in \CC_{br}(\AAA)$. This follows since $M$ is associative and preserves Cauchy nets.
\end{proof}

\begin{definition}\label{twistedquiver}
Let $\AAA$ be a topological $k$-quiver.
A \emph{quiver of twisted objects over $\AAA$} is by definition a quiver $\AAA \subset \Tw(\AAA) \subset \Twfree(\AAA)$ such that for the canonical $\delta \in \CC^1(\Tw(\AAA))$
the sequence $(\sum_{k = 0}^{m}\delta^{\otimes k})_{m\geq 0}$ converges in $\hat{B}\CC$ to a unique element 
$$e^{\delta} = \sum_{k = 0}^{\infty}\delta^{\otimes k}$$
\end{definition}

\begin{remark}
We noticed that the same suggestive exponential notation is used in \cite{fukaya}.
\end{remark}

\begin{proposition}
Let $\AAA$ be a $k$-quiver and consider $\AAA \subset \Tw(\AAA) \subset \Twfree(\AAA)$. The following are equivalent:
\begin{enumerate}
\item $\Tw(\AAA)$ is a quiver of twisted objects over $\AAA$ where $\AAA$ is endowed with the discrete topology.
\item $\Tw(\AAA)$ is a quiver of locally nilpotent twisted objects over $\AAA$.
\end{enumerate}
\end{proposition}

\begin{proof}
By definition of the completion, the sequence converges in $\hat{B}\CC$ if and only if for every $\phi \in \CC_{br}(\AAA)$, the sequence $(\sum_{k=0}^m\phi\{\delta^{\otimes k}\})_{m \geq 0})$ converges for the ``pointwise discrete'' topology $\ttt_0$ on $B\CC$. By definition of this topology, this means that for every $(f_n, \dots ,f_1)$ and $i \in I$ as in Definition \ref{deflocnilptw}, there exists an $m_0$ such that the general term $((\sum_{k=0}^m\phi\{\delta^{\otimes k}\}(f_n, \dots ,f_1)(i))$ becomes constant for $m \geq m_0$. This is clearly equivalent to the fact that the expressions $\phi\{\delta^{\otimes k}\}(f_n, \dots ,f_1)(i)$ become zero for $k \geq m_0$.
\end{proof}

\subsection{Transport of $A_{[0,\infty[}$-structures to $\Tw(\AAA)$}\label{partrans}
Let $\AAA$ be a topological quiver and consider the inclusion $\AAA \subset \Tw(\AAA)$ of $\AAA$ into a quiver of twisted objects over $\AAA$ (in particular, $\Tw(\AAA)$ can be a quiver of locally nilpotent twisted objects over an arbitrary quiver $\AAA$). Let $$\delta \in \CC^1(\Tw(\AAA))$$ be the canonical Hochschild cochain of $\Tw(\AAA)$.

\begin{proposition}\label{bracesect}
The canonical projection $\pi: \CC(\Tw(\AAA)) \lra \CC(\AAA)$ has a brace algebra section
\begin{equation}\label{eqbracesect}
\mathrm{embr}_{\delta}: \CC(\AAA) \lra \CC(\Tw(\AAA)): \phi \longmapsto \sum_{m = 0}^{\infty}\phi\{\delta^{\otimes m}\}
\end{equation}
For $M = (M, \delta_M) \in \Tw(\AAA)$ and $\phi \in \CC^p(\AAA)$, the component $\phi_M = (\mathrm{embr}_{\delta}(\phi))_M \in \Tw(\AAA)(M,M)^1$ is given by
\begin{equation}\label{eqalpha}
\phi_M = \sum_{m=0}^{\infty}(-1)^{\alpha}\phi_m(\delta_M^{\otimes m})
\end{equation} 
with $\alpha = m((p-m) + (m-1)/2)$.
\end{proposition}

 \begin{remark}
 If $\mathrm{char}(k) = 0$, the map $\mathrm{embr}_{\delta}$ is the Lie morphism $e^{[-,\delta]}$.
 \end{remark}
 
\begin{proof}
According to Definition \ref{twistedquiver}, we dispose of an element $e^{\delta} = \sum_{k=0}^{\infty}\delta^{\otimes k} \in \hat{B}\CC$.
We define $\mathrm{embr}_{\delta}$ to be the restriction of the morphism
\begin{equation}\label{hatMe} 
\hat{M}(-,e^{\delta}): B\CC(\AAA) \lra B\CC(\Tw(\AAA))
\end{equation}
which exists by \S \ref{partopo}. In particular, the right hand side of (\ref{eqbracesect}) should be read as a pointwise series, i.e. for $(f_n, \dots, f_1) \in \Tw(\AAA)(M_{n-1},M_n) \otimes \dots \otimes \Tw(\AAA)(M_0,M_1)$ where $M_0 = \oplus_{i \in I}\Sigma^{\alpha_i}A_i$ and $M_n = \oplus_{j \in J}\Sigma^{\beta_j}B_j$, we have $$((\sum_{m = 0}^{\infty}\phi\{\delta^{\otimes m}\})(f_n, \dots, f_1))_{ji} = \sum_{m = 0}^{\infty}((\phi\{\delta^{\otimes m}\}(f_n, \dots, f_1))_{ji})$$
and the right hand side converges for the topology of $\AAA$.
Next we verify that $(\ref{hatMe})$ is a morphism of algebras, i.e. preserves the multiplication $M$.
Consider $\phi, \psi \in B\CC_{br}(\AAA)$. We have $$\hat{M}(\hat{M}(\phi, \psi), e^{\delta}) = \hat{M}(\phi, \hat{M}(\psi, e^{\delta})) = \hat{M}(\hat{M}(\phi, e^{\delta}), \hat{M}(\psi, e^{\delta}))$$
where we used associativity of $M$, continuity of (\ref{eqcont1}) and (\ref{eqcont2}) and the fact that $\hat{M}(e^{\delta}, \psi) = \psi$.
Finally, the statement \eqref{eqalpha} follows from Lemma \ref{lembr}.
\end{proof}
 
 Combining Proposition \ref{bracesect} with Proposition \ref{transfer}, we get:
 \begin{proposition}\label{formulae}
 \begin{enumerate}
 \item If $\mu$ is an $A_{[0, \infty[}$-structure on $\AAA$, then $\mathrm{embr}_{\delta}(\mu)$ is an $A_{[0,\infty[}$-structure on $\Tw(\AAA)$ and $$\mathrm{embr}_{\delta}: (\AAA, \mu) \lra (\Tw(\AAA), \mathrm{embr}_{\delta}(\mu))$$ is a $B_{\infty}$-morphism.
 \item If $\mu = c + d + m$ is a cdg structure on $\AAA$, then $$\mathrm{embr}_{\delta}(\mu) = (c + d\{\delta\} + m\{\delta, \delta\}) + (d + m\{\delta\}) + m$$ is a cdg structure on $\Tw(\AAA)$.
 \item If $\mu = d + m$ is a dg structure on $\AAA$ and $\delta \in \CC^1(\Tw(\AAA))$ satisfies $$d\{\delta\} + m\{\delta, \delta\} = 0$$ then $$\mathrm{embr}_{\delta}(\mu) = (d + m\{\delta\}) + m$$ is a dg structure on $\Tw(\AAA)$.
 \end{enumerate}
 \end{proposition}
From now on, quivers of twisted objects over an $A_{[0,\infty[}$-category $(\AAA,\mu)$ will always be endowed with the $A_{[0,\infty[}$-structure $\mathrm{embr}_{\delta}(\mu)$.

 \begin{remark}
 A similar kind of ``transport'' is used in \cite[\S 6]{lefevre} in order to construct $A_{\infty}$-functor categories.
 \end{remark}

\subsection{Classical twisted complexes}\label{parclass}

We will now discuss how some classical categories of twisted complexes fit in the framework of the previous sections.

 \begin{definition}\label{inftypart}
 Let $\AAA = (\AAA, \mu)$ be an $A_{[0, \infty[}$-category. The \emph{$\infty$-part of $\AAA$} is the full subcategory $\AAA_{\infty} \subset \AAA$ with as objects the $A \in \AAA$ for which $\mu_A \in \AAA(A,A)^2$ is zero.
 \end{definition}

\begin{example}
Let $\AAA$ be an $A_{\infty}$-category and let $\twilnil(\AAA) \subset \Twilnil(\AAA)$ be the quiver with as objects the $(M, \delta_M)$ where $M = \oplus_{i=0}^k\Sigma^{m_i} A_i$ is ``finite''. 
\begin{enumerate}
\item If $\AAA$ is a dg category , then the dg category $\twilnil(\AAA)_{\infty}$ is equivalent to the classical dg category of twisted complexes over $\AAA$ (\cite{bondalkapranov}, \cite{drinfeld}, \cite{keller}). Indeed, the $\infty$-part of $\twilnil(\AAA)$ is its restriction to the objects $(M, \delta_M)$ with $$d\{\delta_M\} + m\{\delta_M, \delta_M\} = 0$$
More generally, $\twilnil(\AAA)_{\infty}$ is equivalent to the $A_{\infty}$-category $\tw(\AAA)$ of twisted objects over $\AAA$ (\cite{lefevre}, and \cite{fukaya} for the algebra case).
\item The dg category $\Twilnil(\AAA)_{\infty}$ is equivalent to the classical dg category of semifree complexes over $\AAA$ (\cite{drinfeld}) which is a dg-model for $D(\AAA)$, i.e. there is an equivalence of triangulated categories $H^0(\Twilnil(\AAA)_{\infty}) \cong D(\AAA)$.
\end{enumerate}
\end{example}

\begin{remark}
We conjecture that for an $A_{\infty}$-category $\AAA$, the $A_{\infty}$-category $\Twilnil(\AAA)_{\infty}$ is an $A_{\infty}$-model for the derived category of $\AAA$, i.e. there is an equivalence of triangulated categories $H^0(\Twilnil(\AAA)_{\infty}) \cong D_{\infty}(\AAA)$, where for definitions of the right hand side, we refer the reader to \cite{lefevre}. The finite version of this result has been obtained in \cite[\S 7.4]{lefevre}.
\end{remark}

\begin{remark}
For an $A_{[0,\infty[}$-category $\AAA$, we have $\twilnil(\AAA)_{\infty} = \twilnil(\AAA_{\infty})_{\infty}$ and similarly for $\Twilnil(\AAA)$. In particular, an $A_{[0,\infty[}$-algebra $A$ with $\mu_0 \neq 0$ has $\Twilnil(\AAA)_{\infty} = 0$. This illustrates the poor ``derivability'', in general, of $A_{[0,\infty[}$-algebras.
\end{remark}

The following theorem, which immediately follows from Proposition \ref{bracesect}, is a refinement of \cite[Theorem 4.4.1]{lowenvandenbergh2}.

\begin{theorem}
Let $\AAA$ be a dg category. Then $\Tw(\AAA) = \Twilnil(\AAA)_{\infty}$ is a dg category which is dg equivalent to the category of semifree dg modules over $\AAA$. The canonical projection $\pi: \CC(\Tw(\AAA)) \lra \CC(\AAA)$ has a $B_{\infty}$-section $$\mathrm{embr}_{\delta}: \CC(\AAA) \lra \CC(\Tw(\AAA)): \phi \longmapsto \sum_{m = 0}^{\infty}\phi\{\delta^{\otimes m}\}$$
which is an inverse in the homotopy category of $B_{\infty}$-algebras. In particular, both $\pi$ and $\mathrm{embr}_{\delta}$ are quasi-isomorphisms. 
\end{theorem}

\subsection{(Pre)complexes over linear categories}\label{exprecompl}

Next we apply Proposition \ref{bracesect} to categories of (pre)complexes. Let $(\AAA, m)$ be a linear category.
Consider the quiver $\Tw_{\mathsf{pre}}(\AAA)$ with as objects 
$$(M = \oplus_{i \in \Z}\Sigma^{i}A_i, \delta_M)$$ 
with $\delta_M \in \Free(\AAA)(M,M)^1$. For another $(N = \oplus_{i \in \Z}\Sigma^{i}B_i, \delta_N) $, since $\AAA$ is concentrated in degree zero, we have $\Tw_{\mathsf{pre}}(\AAA)(M,N)^n =  \prod_{i\in \Z}\AAA(A_i,B_{i-n})$. If we change to cohomological notation $A^i = A_{-i}$, we have
$$\Tw_{\mathsf{pre}}(\AAA)(M,N)^n =  \prod_{i\in \Z}\AAA(A^i,B^{i + n})$$
By Example \ref{exdegzero}, $\Tw_{\mathsf{pre}}(\AAA)$ is a quiver of locally nilpotent twisted objects over $\AAA$.
According to Proposition \ref{formulae}, the corresponding $A_{[0,\infty[}$-structure on $\Tw_{\mathsf{pre}}(\AAA)$ is $\embr_{\delta}(m) = m\{\delta, \delta\} + m\{\delta\} + m$ with
$$m\{ \delta, \delta \} = - \delta^2$$
and $$m \{ \delta \} = m(\delta \otimes 1 - 1 \otimes \delta)$$
Hence $\Tw_{\mathsf{pre}}(\AAA)$ is precisely the cdg category $\PCom(\AAA)$ of precomplexes of $\AAA$-objects of example \ref{pcomexample}.
The category $\Tw_{\mathsf{com}}(\AAA) = \Tw_{\mathsf{pre}}(\AAA)_{\infty}$ is the dg category $\Com(\AAA)$ of complexes of $\AAA$-objects.

Consider the inclusions
$$\AAA \subset \Com^+(\AAA) \subset \Com(\AAA)$$
The following is implicit in \cite{lowenvandenbergh2}:

\begin{proposition}
The canonical projection $\pi: \CC(\Com^+(\AAA)) \lra \CC(\AAA)$ is a $B_{\infty}$-quasi-isomorphism.
\end{proposition}

\begin{proof}
Consider the canonical morphisms
$$\AAA^{\op} \lra \Com^-(\AAA^{\op}) \lra \Com^-(\Mod(\AAA^{\op})) \lra \Com(\Mod(\AAA^{\op})) = \Mod_{\mathrm{dg}}(\AAA^{\op})$$ 
A complex in $\Com^-(\AAA^{\op})$ gets mapped to a cofibrant object in $\Mod_{\mathrm{dg}}(\AAA^{\op})$. Consequently, by \cite[Theorem 4.4.1]{lowenvandenbergh2}, the first map induces a $B_{\infty}$-quasi-isomorphism. The result follows since $\pi$ is induced by the opposite of this map.
\end{proof}

\begin{theorem}\label{anotherembr}
The canonical projection $\pi: \CC(\PCom(\AAA)) \lra \CC(\AAA)$ has a $B_{\infty}$-section 
$$\mathrm{embr}_{\delta}: \CC(\AAA) \lra \CC(\PCom(\AAA)): \phi_n \longmapsto \sum_{m=0}^n\phi_n\{\delta^{\otimes m}\}$$
The restrictions of both maps to $\CC(\Com^+(\AAA))$ are inverse isomorphisms in the homotopy category of $B_{\infty}$-algebras. In particular, they are both quasi-isomorphisms.
\end{theorem}

\subsection{Abelian categories}\label{parabcat}
The results of the previous section have an immediate application to abelian categories.
Let $\aaa$ be an abelian category. In \cite{lowenvandenbergh2}, the Hochschild complex of $\aaa$ is defined to be 
$$\CC_{\mathrm{ab}}(\aaa) = \CC(\Inj(\Ind(\aaa))$$

Let $\aaa$ be an abelian category with enough injectives and put $\III = \Inj(\aaa)$. By \cite[Theorem 6.6]{lowenvandenbergh2}, we have
$$\CC_{\mathrm{ab}}(\aaa) \cong \CC(\III)$$
and it will be convenient to actually take this as definition of $\CC_{\mathrm{ab}}(\aaa)$. 

The dg category $\Com^+(\III)$ of bounded below complexes of injectives is a dg model for the bounded below derived category $D^+(\aaa)$ of $\aaa$, whence the notation $D^+_{\mathrm{dg}}(\aaa) = \Com^+(\III)$.
In the spirit of \cite{keller6}, put $\CC_{\mathrm{ex}}(\aaa) = \CC(D^+_{\mathrm{dg}}(\aaa))$.
With $\AAA = \III$, Theorem \ref{anotherembr} now yields:

\begin{theorem}\label{anotherembr2}
The canonical projection $\pi: \CC(\Com^+(\III)) \lra \CC(\III)$ has a $B_{\infty}$-section 
$$\mathrm{embr}_{\delta}: \CC(\III) \lra \CC(\Com^+(\III)): \phi_n \longmapsto \sum_{m=0}^n\phi_n\{\delta^{\otimes m}\}$$
which is an inverse in the homotopy category of $B_{\infty}$-algebras. In particular, both $\pi$ and $\mathrm{embr}_{\delta}$ are quasi-isomorphisms establishing
$\CC_{\mathrm{ab}}(\aaa) \cong \CC_{\mathrm{ex}}(\aaa)$.
\end{theorem}

\section{Deformations}\label{pardefo}

This chapter consists largely of applications of Theorem \ref{anotherembr}. We first recall some facts on deformations and the graded centre enabling us to define, in \S \ref{parchar}, the characteristic dg morphism of a linear category, and to show its relation to deformation theory in Theorem \ref{maintheorem}. The remainder of the chapter is devoted to some applications to deformations of (enhanced) derived categories of abelian categories.

Throughout we focus on \emph{first order} deformations, i.e. deformations along $k[\epsilon] \lra k$, since they are in the most direct correspondence with Hochschild cohomology. All definitions can be given for arbitrary deformations, and in the classical setting of an Artin local algebra $R$ over a field $k$ of characteristic zero (with maximal ideal $m$), the deformation theory is governed by the Maurer-Cartan equation in the Hochschild complex (tensored by $m$).

From now on, $k$ will be a field.

\subsection{Deformations of linear and abelian categories}

The deformation theory of linear and abelian categories was developed in \cite{lowenvandenbergh1} as a natural extension of Gerstenhaber's deformation theory of algebras \cite{gerstenhaber}. In this section we recall the main definitions. For a commutative ring $R$, let $\cat(R)$ denote the (large) category of $R$-linear categories. The forgetful functor $\cat(k) \lra \cat(k[\epsilon])$ has the left adjoint
$$k \otimes_{k[\epsilon]} -: \cat(k[\epsilon]) \lra \cat(k)$$
and the right adjoint
$$\Hom_{k[\epsilon]}(k,-): \cat(k[\epsilon]) \lra \cat(k)$$
where $\Hom_{k[\epsilon]}$ denotes the category of $k[\epsilon]$-linear functors. Clearly, for $\bbb \in \cat(k[\epsilon])$, there is a canonical inclusion functor $\Hom_{k[\epsilon]}(k,\bbb) \lra \bbb$ identifying $\Hom_{k[\epsilon]}(k,\bbb)$ with the full subcategory of objects $B \in \bbb$ for which $\epsilon: B \lra B$ is equal to zero.

In \cite{lowenvandenbergh1}, a notion of flatness for \emph{abelian} $R$-linear categories is defined which is such that an $R$-linear category $\AAA$ is flat (in the sense that it has $R$-flat hom-modules) if and only if the module category $\Mod(\AAA)$ is flat as an abelian category.

\begin{definition}
\begin{enumerate}
\item  Let $\AAA$ be a $k$-linear category. A first order \emph{linear deformation} of $\AAA$ is a flat $k[\epsilon]$-linear category $\BBB$ together with an isomorphism $k \otimes_{k[\epsilon]} \BBB \cong \AAA$ in $\cat(k)$.
\item Let $\aaa$ be an abelian $k$-linear category. A first order \emph{abelian deformation} of $\aaa$ is a flat abelian $k[\epsilon]$-linear category $\bbb$ together with an equivalence of categories $\aaa \cong \Hom_{k[\epsilon]}(k, \bbb)$.
\end{enumerate}
\end{definition}

We will denote the natural groupoids of linear deformations of $\AAA$ and of abelian deformations of $\aaa$ by $\Def_{\AAA}(k[\epsilon])$ and $\text{ab}-\Def_{\aaa}(k[\epsilon])$ respectively. In \cite{lowenvandenbergh1}, the notations $\ddef^s_{\AAA}$ and $\Def_{\aaa}$ are used and the terminology \emph{strict} deformation is used in the linear case.

The following proposition extends the well known result for algebras:

\begin{proposition}
Let $\AAA$ be a $k$-linear category.
There is a map $$Z^2\CC(\AAA) \lra \Ob(\Def_{\AAA}(k[\epsilon]))$$ which induces a bijection
$$HH^2(\AAA) \lra \Sk(\Def_{\AAA}(k[\epsilon]))$$
\end{proposition}

\begin{proof}
Consider $\phi \in Z^2\CC(\AAA)$. The cocycle $\phi$ describes the corresponding linear deformation of $(\AAA, m)$ in the following way. Consider the quiver $\AAA[\epsilon] = k[\epsilon] \otimes_k \AAA$ over $k[\epsilon]$. The linear deformation of $\AAA$ is $\AAA_{\phi}[\epsilon] = (\AAA[\epsilon], m + \phi \epsilon)$.
\end{proof}

Finally we mention the following fundamental result of \cite{lowenvandenbergh2}, where the Hochschild cohomology of the abelian category $\aaa$ is as defined in \S \ref{parabcat}.

\begin{proposition}
Let $\aaa$ be a $k$-linear abelian category. There is a bijection 
$$HH^2_{ab}(\aaa) \lra \mathrm{ab}-\Def_{\aaa}(k[\epsilon])$$ 
\end{proposition}

\subsection{The centre of a graded category}\label{parcentre}

We recall the definition of the centre of a graded category (see also \cite[\S 3]{buchweitzflenner3}).

\begin{definition}
Let $\AAA$ be a graded category. The \emph{centre of $\AAA$} is the centre of $\AAA$ as a category enriched in $G(k)$, i.e. $$\ZZZ(\AAA) = \Hom(1_{\AAA}, 1_{\AAA})$$
where $1_{\AAA}: \AAA \lra \AAA$ is the identity functor and $\Hom$ denotes the graded module of graded natural transformations. 
\end{definition}

\begin{remark}
Explicitely, an element in $\ZZZ(\AAA)$ is given by an element $(\zeta_A)_A \in \prod_{A \in \AAA}\AAA(A,A)$ with the naturality property that for all $A, A' \in \AAA$, the following diagram commutes:
$$\xymatrix{{\AAA(A,A')} \ar[r]^-{\zeta \otimes 1} \ar[d]_-{1 \otimes \zeta} & {\AAA(A',A') \otimes \AAA(A,A')} \ar[d]^m \\ {\AAA(A,A') \otimes \AAA(A,A)} \ar[r]_-{m} & {\AAA(A,A')}}$$
In other words, for $f \in \AAA(A,A')$, $$\zeta_{A'}f = (-1)^{|f||\zeta|}f\zeta_A$$
\end{remark}

\begin{remark}
Let $\ttt$ be a \emph{suspended} linear category with suspension $\Sigma_{\ttt}: \ttt \lra \ttt$. There is an associated graded category $\ttt_{gr}$ with $\ttt_{gr}(T,T')^n = \ttt(T, \Sigma^n_{\ttt} T')$ and the \emph{graded centre of \ttt} is the centre of the graded category $\ttt_{gr}$.
If $\TTT$ is an exact dg category with associated triangulated category $\ttt = H^0\TTT$, we have $\ttt_{gr} = H^{\ast}\TTT$.
\end{remark}

\subsection{The characteristic dg morphism}\label{parchar}
It is well known that for a $k$-algebra $A$, there is a characteristic morphism of graded commutative algebras from the Hochschild cohomology of $A$  to the graded centre of the derived category $D(A)$.
This morphism is determined by the maps, for $M \in D(A)$,
$$M \otimes^{L}_A -: HH^{\ast}_k(A) \cong \Ext^{\ast}_{A^{\op} \otimes A}(A,A) \lra \Ext^{\ast}_A(M,M)$$
The characteristic morphism occurs for example in the theory of support varieties (\cite{buchweitzavramov}, \cite{erdmann}, \cite{solberg}).
Recently, Buchweitz and Flenner proved the existence of a characteristic morphism in the context of  \emph{morphisms} of schemes or analytic spaces (\cite{buchweitzflenner3}).

In \cite{lowenvandenbergh2}, it is observed that a characteristic morphism also exists for abelian categories. 
Let $\aaa$ be an abelian category with enough injectives, $\III = \Inj(\aaa)$ and $\Com(\III)$ the dg category of complexes of injectives. As asserted in Proposition \ref{projzero}, there is a morphism of differential graded objects
$$\pi_0: \CC(\Com(\III)) \lra \prod_{E \in \Com(\III)} \Com(\III)(E,E)$$
Taking cohomology of $\pi_0$ (where we restrict to $\Com^+(\III)$) and composing with the isomorphisms $HH^{\ast}_{\mathrm{ab}}(\aaa) \cong HH^{\ast}_{\mathrm{ex}}(\aaa)$ of Theorem \ref{anotherembr2}, we obtain the characteristic morphism
$$\chi_{\aaa}: HH^{\ast}_{\mathrm{ab}}(\aaa) \lra \ZZZ^{\ast}D^+(\aaa)$$
Using the $B_{\infty}$-section of Proposition \ref{anotherembr}, we can actually lift the characteristic morphism to the level of dg objects. In fact we can construct this lifted characteristic morphism for an arbitrary $k$-linear category $\AAA$ instead of $\III$.

\begin{definition}\label{defchardg}
Let $\AAA$ be a $k$-linear category. The \emph{characteristic dg morphism}
$$\CC(\AAA) \lra \prod_{C \in \Com(\AAA)} \Com(\AAA)(C,C)$$
is the composition of the $B_{\infty}$-morphism $\CC(\AAA) \lra \CC(\Com(\AAA))$ of Theorem \ref{anotherembr} and the projection on the zero part of Proposition \ref{projzero}.
Taking cohomology, we obtain the \emph{characteristic morphism}
$$HH^{\ast}(\AAA) \lra \ZZZ^{\ast}K(\AAA)$$
where $K(\AAA)$ is the homotopy category of complexes of $\AAA$-objects.
\end{definition}
In the next section we will interpret the characteristic morphism in terms of deformation theory.

\subsection{The characteristic morphism and obstructions}\label{parcharobs}
Let $\AAA$ be a $k$-linear category. In \cite{lowen2}, an obstruction theory is established for deforming objects of the homotopy category $K(\AAA)$. Let $c \in Z^2\CC(\AAA)$ be a Hochschild cocycle and $\AAA_c[\epsilon]$ the corresponding linear deformation. Consider the functor
$$k \otimes_{k[\epsilon]} - : K(\AAA_c[\epsilon]) \lra K(\AAA)$$
and consider $C \in K(\AAA)$. We will say that a \emph{(homotopy) $c$-deformation} of $C$ is a lift of $C$ along $k \otimes_{k[\epsilon]} -$. 
According to \cite[Theorem 5.2]{lowen2}, first order $c$-deformations of $C$ are governed by an obstruction theory involving $K(\AAA)(C,C[2])$ and $K(\AAA)(C,C[1])$.
In particular, the obstruction against $c$-deforming $C$ is an element $o_c \in K(\AAA)(C,C[2])$ depending on $c$, whereas $K(\AAA)(C,C[2])$ itself is independent of $c$. In the remainder of this section we show that the way in which the obstruction $o_c$ depends on $c$ is encoded in the characteristic morphism. 

\begin{theorem}\label{maintheorem}
Let $\AAA$ be a linear category and consider the characteristic morphism
$$\chi_{\AAA}: HH^{2}(\AAA) \lra \ZZZ^{2}K(\AAA)$$
We have $$\chi_{\AAA}(c) = (o_C)_{C \in K(\AAA)}$$
where $o_C \in K(\AAA)(C,C[2])$ is the obstruction to $c$-deforming $C$ into an object of $K(\AAA_c[\epsilon])$.
\end{theorem}

\begin{proof}
Let $\bar{\chi}_{\AAA}$ be the characteristic dg morphism $\CC^2(\AAA) \lra \prod_{C \in \Com(\AAA)}\Com(\AAA)^2(C,C)$ enhancing $\chi_{\AAA}$. Consider $C = (C, \delta_C) \in \Com(\AAA)$ and $\phi \in \CC^2(\AAA)$. According to Theorem \ref{anotherembr}, we have $$(\bar{\chi}_{\AAA}(\phi))_C = - \phi(\delta_C, \delta_C)$$ 
According to \cite[Theorems 3.8, 4.1]{lowen2}, $[\phi(\delta_C, \delta_C)]$ is the obstruction to $c$-deforming $C$.
\end{proof}

\begin{corollary}
Let $\aaa$ be an abelian category with enough injectives. The characteristic morphism
$$\chi_{\aaa}: HH^{2}_{\mathrm{ab}}(\aaa) \lra \ZZZ^{2}D^+(\aaa)$$
satisfies $$\chi_{\aaa}(c) = (o_C)_{C \in D^+(\aaa)}$$
where $o_C \in \Ext^2_{\aaa}(C,C)$ is the obstruction to deforming $C$ into an object of $D^+(\aaa_c)$.
For $C \in D^b(\aaa)$, this is equally the obstruction to deforming $C$ into an object of $D^b(\aaa_c)$.
\end{corollary}

\begin{proof}
This easily follows from Theorem \ref{maintheorem} since $D^+(\aaa) \cong K^+(\Inj(\aaa))$. The  last statement follows from \cite[\S 6.3]{lowen2}.
\end{proof}

\subsection{$A_{[0,\infty[}$-deformations}
In this section, we discuss the sense in which the Hochschild cohomology of an $A_{[0,\infty[}$-category $\AAA$ describes its first order $A_{[0,\infty[}$-deformations. The easiest (though not necessarily the best, see Remark \ref{remcdgdefo} and \S \ref{parderdef}) deformations to handle are those with fixed set of objects. Note that the ``flatness'' automatically imposed in this definition is graded freeness, which does \emph{not} imply cofibrancy in the dg-case!

\begin{definition}\label{cdgdefo}
Let $\AAA$ be a $k$-linear $A_{[0,\infty[}$-category where, in the entire definition, $A_{[0,\infty[}$- can be replaced by $A_{\infty}$-, cdg, dg or blanc.
\begin{enumerate}
\item A  \emph{first order $A_{[0,\infty[}$-deformation of $\AAA$} is a structure of $k[\epsilon]$-linear $A_{[0,\infty[}$-category on a $k[\epsilon]$-quiver $\BBB \cong k[\epsilon] \otimes_k \AAA$, such that its reduction to $\AAA$ coincides with the $A_{[0,\infty[}$-structure of $\AAA$ (in other words, the canonical $k \otimes_{k[\epsilon]} \BBB \cong \AAA$ is an $A_{[0,\infty[}$-isomorphism).
\item A \emph{partial first order $A_{[0,\infty[}$-deformation of $\AAA$} is an $A_{[0,\infty[}$-deformation of $\AAA'$ for some full $A_{[0,\infty[}$-subcategory $\AAA' \subset \AAA$. 
\item  Let $\BBB$ and $\BBB'$ be (partial) deformations of $\AAA$. An \emph{isomorphism of (partial) deformations} is an isomorphism $g: \BBB \lra \BBB'$ of $A_{[0,\infty[}$-categories, of which the reduction to $\AAA$ (resp. $\AAA'$ in case of partial deformations) is the identity morphism. A \emph{morphism of partial deformations} is an isomorphism of deformations between $\BBB$ and a full $A_{[0,\infty[}$-subcategory of $\BBB'$.
\item A partial deformation $\BBB$ of $\AAA$ is called \emph{maximal} if every morphism $\BBB \lra \BBB'$ of partial deformations is an isomorphism.
\item The groupoid $A_{[0,\infty[}-\Def_{\AAA}(k[\epsilon])$ has as objects first order $A_{[0,\infty[}$-deformations of $\AAA$. Its morphisms are isomorphisms of deformations.
\item The category $A_{[0,\infty[}-\mathrm{pDef}_{\AAA}(k[\epsilon])$ has as objects first order partial $A_{[0,\infty[}$-deformations of $\AAA$. Its morphisms are morphisms of partial deformations.
\item The groupoid $A_{[0,\infty[}-\mathrm{mpDef}_{\AAA}(k[\epsilon])$ has as objects maximal partial $A_{[0,\infty[}$-deformations of $\AAA$. Its morphisms are isomorphisms of partial deformations.
\item The groupoid $MC_{\AAA}(k[\epsilon])$ has as objects $Z\CC^2(\AAA)$. For $c, c' \in Z\CC^2(\AAA)$, a morphism $c \lra c'$ is an element $h \in \CC^1(\AAA)$ with $d(h) = c' -c$.
\end{enumerate}
\end{definition}

\begin{proposition}\label{mcprop1}
Let $\AAA$ be a $k$-linear $A_{[0,\infty[}$-category.
\begin{enumerate}
\item
 There is an equivalence of categories 
$$MC_{\AAA}(k[\epsilon]) \lra A_{[0,\infty[}{-}\Def_{\AAA}(k[\epsilon])$$
\item Consequently, there is a bijection
$$HH^2(\AAA) \lra \Sk(A_{[0,\infty[}{-}\Def_{\AAA}(k[\epsilon]))$$
\end{enumerate}
\end{proposition}

\begin{proof}
Let $\mu$ be the $A_{[0,\infty[}$-structure on $\AAA$ and consider $\phi \in Z^2\CC(\AAA)$. The image of $\phi$ is the $A_{[0,\infty[}$-category $A_{\phi}[\epsilon] = (\AAA[\epsilon],  \mu + \phi\epsilon)$.
To see that $\mu + \phi\epsilon$ is an $A_{[0,\infty[}$-structure, it suffices to compute
$\mu + \phi\epsilon\{\mu + \phi\epsilon\} = \mu\{\mu\} + [\mu\{\phi\} + \phi\{\mu\}]\epsilon$ which is zero since $\mu$ is an $A_{[0,\infty[}$-structure and  $\phi$ is a Hochschild cocycle.
Next consider a morphism of cocycles $h: \phi \lra \phi'$. The image of $h$ is the morphism
$1 + h\epsilon: B\AAA_{\phi}[\epsilon] \lra B\AAA_{\phi'}[\epsilon]$. The identity $d(h) = \phi' - \phi$ easily implies the compatibility of $1 + h\epsilon$ with the respective $A_{[0,\infty[}$-structures.
\end{proof}

\begin{definition}
Consider a $k$-linear $A_{\infty}$-category $\AAA$ and $\phi \in HH^2(\AAA)$. The $\phi-\infty$-part of $\AAA$ is the full subcategory $\AAA_{\phi-\infty} \subset \AAA$ with
$$\Ob(\AAA_{\phi-\infty}) = \{ A \in \AAA \,\, |\, \, 0 = H^2(\pi_0)(\phi) \in H^2(\AAA(A,A))\}$$
where $\pi_0$ is as in \S \ref{parproj}.
\end{definition}

\begin{example}
Consider a linear category $\AAA$ and $\phi \in HH^2(\AAA)$. Put $\phi' = [\embr_{\delta}](\phi) \in HH^2(\Com(\AAA))$. We have
$$\Ob(\Com(\AAA)_{\phi'-\infty}) = \{ C \in \Com(\AAA) \, \, | \, \, 0 = \chi_{\AAA}(\phi)_C \in K(\AAA)(C,C[2]) \}$$
\end{example}

\begin{proposition}\label{mcprop2}
Let $\AAA$ be a $k$-linear $A_{\infty}$-category.
\begin{enumerate}
\item There is a morphism
$$A_{[0,\infty[}-\Def_{\AAA}(k[\epsilon]) \lra A_{\infty}-\mathrm{pDef}_{\AAA}(k[\epsilon]): \BBB \longmapsto \BBB_{\infty}$$
where $\BBB_{\infty}$ is as in Definition \ref{inftypart}.
\item There is a morphism
$$HH^2(\AAA) \lra \Sk(A_{\infty}-\mathrm{pDef}_{\AAA}(k[\epsilon]))$$
mapping $\phi \in HH^2(\AAA)$ to an $A_{\infty}$-deformation of $\AAA_{\phi-\infty} \subset \AAA$.
\end{enumerate}
\end{proposition}

\begin{proof}
For (ii), let $\bar{\phi} \in Z^2\CC(\AAA)$ be a Hochschild cocycle with $[\bar{\phi}] =\phi$ and let $\bar{\phi}'$ be its restriction to $Z^2\CC(\AAA_{\phi-\infty})$. Then $(\bar{\phi}')_0 \in \prod_{A \in \AAA_{\phi-\infty}}\AAA^2(A,A)$ is a coboundary, hence there exists $h \in \prod_{A \in \AAA_{\phi-\infty}}\AAA^1(A,A)$ with $d_{\AAA}(h) = (\bar{\phi}')_0$. If we consider $h$ as an element of $\CC^1(\AAA)$, then $\bar{\phi}'' = \bar{\phi}' - d(h)$ is a representative of $\phi$ with $(\bar{\phi}' + d(h))_0 = 0$. Consequently, $(A_{\phi-\infty})_{\bar{\phi}''}[\epsilon]$ is a partial $A_{\infty}$-deformation of $\AAA$ corresponding to $\phi$.
\end{proof}

\subsection{Deformations of categories of (pre)complexes}
Let $\AAA$ be a $k$-linear category.
In this section we use Theorem \ref{anotherembr} to associate to a linear deformation of $\AAA$, a cdg deformation of the cdg category $\PCom(\AAA)$ of precomplexes of $\AAA$-objects, and a partial dg deformation of the dg category $\Com(\AAA)$ of complexes of $\AAA$-objects.

Combining Theorem \ref{anotherembr} and Proposition \ref{mcprop1} we obtain a functor $$MC_{\AAA}(k[\epsilon]) \lra MC_{\PCom(\AAA)}(k[\epsilon]) \lra A_{[0,\infty[}-\Def_{\PCom(\AAA)}(k[\epsilon])$$ factoring through a ``realization'' functor
$$R: MC_{\AAA}(k[\epsilon]) \lra \text{cdg}-\Def_{\PCom(\AAA)}(k[\epsilon])$$
whose restriction to cdg$-\Def_{\Com^+(\AAA)}(k[\epsilon])$ is an equivalence.
Similarly, using Proposition \ref{mcprop2}(2), there is a map
$$\rho': HH^2(\AAA) \lra HH^2(\Com^+(\AAA)) \lra \Sk(\text{dg}-\mathrm{pDef}_{\Com^+(\AAA)}(k[\epsilon]))$$

\begin{theorem}\label{precomplexes}
Consider $\phi \in Z^2\CC(\AAA)$ and the corresponding linear deformation $\AAA_{\phi}[\epsilon]$. 
\begin{enumerate}
\item The cdg deformation $R(\phi)$ of $\PCom(\AAA)$ is (isomorphic to) the subcategory $\PCom_{triv}(\AAA_{\phi}[\epsilon]) \subset \PCom(\AAA_{\phi}[\epsilon])$ consisting of the ``trivial'' precomplexes $\bar{C} = k[\epsilon] \otimes_k C$ for $C \in \PCom(\AAA)$ .
\item For every collection of precomplexes $\Gamma = \{\bar{C}\}_{C \in \PCom(\AAA)}$ where $k \otimes_{k[\epsilon]} \bar{C} = C$, the subcategory $\PCom_{\Gamma}(\AAA_{\phi}[\epsilon]) \subset \PCom(\AAA_{\phi}[\epsilon])$ spanned by $\Gamma$ is a cdg deformation of $\PCom(\AAA)$ which is isomorphic to $R(\phi)$.
\item For every collection of complexes $\Gamma = \{\bar{C}\}_{C \in \Com^+(\AAA)_{\phi-\infty}}$ where $k \otimes_{k[\epsilon]} \bar{C} = C$, the subcategory $\Com^+_{\Gamma}(\AAA_{\phi}[\epsilon]) \subset \Com^+(\AAA_{\phi}[\epsilon])$ spanned by $\Gamma$ is a maximal partial dg deformation of $\Com^+(\AAA)$ representing $\rho'([\phi])$.
\end{enumerate}
Consequently, $\rho'$ factors over an injection 
$$\rho: HH^2(\AAA) \lra \Sk(\mathrm{dg}-\mathrm{mpDef}_{\Com^+(\AAA)}(k[\epsilon]))$$
The image consists of those maximal partial  dg deformations that are dg deformations of some $\AAA'$ with $\AAA \subset \AAA' \subset \Com^+(\AAA)$.
\end{theorem}

\begin{remark}
According to Theorem \ref{precomplexes} iii), the part of $\Com^+(\AAA)$ that ``dg-deforms'' with respect to $\phi \in HH^2(\AAA)$ is spanned by the objects $$\{ C \in \Com^+(\AAA)\,\, | \,\, 0 = \chi_{\AAA}(\phi)_C \in K(\AAA)(C,C[2]) \}$$
\end{remark}

\begin{remark}\label{remcdgdefo}
Morally, Theorem \ref{precomplexes} ii) suggests that we may consider $\PCom(\AAA_{\phi}[\epsilon])$ as a representative of the class of  cdg deformations of $\PCom(\AAA)$ corresponding to the element $[\phi] \in HH^2(\AAA)$. To make this statement mathematically true, one needs a somewhat more relaxed notion of deformation (and isomorphism) in which the object set is not necessarily preserved. Clearly, the statement is true for any such notion of which Definition \ref{cdgdefo} with the isomorphisms relaxed to fully faithful morphisms that are surjective on objects is a special case. In the same spirit, iii) suggests that we may  consider $\Com^+(\AAA_{\phi}[\epsilon])$ as a representative of the class of maximal partial dg deformations of $\Com^+(\AAA)$ corresponding to $[\phi]$.
\end{remark}

\begin{proof}
There is a canonical morphism of $k[\epsilon]$-quivers $$F: \PCom(\AAA_{\phi}[\epsilon]) \lra k[\epsilon] \otimes_k \PCom(\AAA)$$ defined in the following manner.
A precomplex $\bar{C}$ of $\AAA_{\phi}[\epsilon]$-objects gets mapped to $C = k \otimes_{k[\epsilon]} \bar{C} \in \PCom(\AAA)$. For two precomplexes $\bar{C}$ and $\bar{D}$, $\PCom(\AAA_{\phi}[\epsilon])^n(\bar{C},\bar{D}) = \prod_{i \in \Z}\AAA_{\phi}[\epsilon](C^i, D^{i+n}) \cong k[\epsilon] \otimes \PCom(\AAA)^n(C,D)$. This defines $F$. From now on we will tacitly use $F$ to identify the left and the right hand side.

Let us denote the composition of $\AAA$ by $m$. By definition, the composition of $\AAA_{\phi}[\epsilon]$ is $m + \phi\epsilon$. Write $\delta$ for the predifferentials in $\PCom(\AAA)$ and $\bar{\delta} = \delta + \delta'\epsilon$ for the predifferentials in (a full subcategory of) $\PCom(\AAA_{\phi}[\epsilon])$. By Examples \ref{pcomexample}, \ref{exprecompl}, the cdg structure on $\PCom(\AAA)$ is given by $\mu = m\{\delta, \delta\} + m\{\delta\} + m$ and the cdg structure on $\PCom(\AAA_{\phi}[\epsilon])$ is given by $\tilde{\mu} = (m + \phi\epsilon)\{\delta +\delta'\epsilon, \delta + \delta' \epsilon\} + (m + \phi \epsilon)\{\delta + \delta'\epsilon\} + m$. This expression can be rewritten as
$$\tilde{\mu}_0 = m\{\delta, \delta\} + [m\{\delta, \delta'\} + m\{\delta', \delta\} + \phi\{\delta, \delta\}]\epsilon$$ 
$$\tilde{\mu}_1 = m\{\delta\} + [m\{\delta'\} + \phi\{\delta\}]\epsilon$$
$$\tilde{\mu}_2 = m + \phi\epsilon$$
On the other hand, we have $\embr_{\delta}(\phi) = \phi\{\delta,\delta\} + \phi\{\delta\} + \phi$ so the cdg structure on $\PCom(\AAA)_{\embr_{\delta}(\phi)}[\epsilon]$ is $\bar{\mu}$ with 
$$\bar{\mu}_0 = m\{\delta, \delta\} + \phi\{\delta, \delta\}\epsilon$$
$$\bar{\mu}_1 = m\{\delta\} + \phi\{\delta\}\epsilon$$
$$\bar{\mu}_2 = m + \phi\epsilon$$
Comparing $\tilde{\mu}$ and $\bar{\mu}$, it becomes clear that on trivial precomplexes (where $\delta' = 0$), they coincide. This already proves i).
To produce a deformation isomorphic to $\bar{\mu}$, it is, by Proposition \ref{mcprop1}, allowed to change $\embr_{\delta}(\phi)$ up to a Hochschild coboundary. For a collection $\Gamma$ as in (2), consider the corresponding $\delta' \in \CC^1(\PCom(\AAA))$. Using Definition \ref{bracedef} and the definition of the Hochschild differential $d$ (\S \ref{parhochainf}), it becomes clear that 
$$\tilde{\mu} = \bar{\mu} + d(\delta')\epsilon$$
thus proving ii).

Now consider a collection $\Gamma$ of complexes as in iii). Obviously $\Com_{\Gamma}^+(\AAA_{\phi}[\epsilon])$ defines a dg deformation of $\Com^+(\AAA)_{\phi-\infty}$ hence a partial deformation of $\Com^+(\AAA)$. By the reasoning above, dg deformations of $\AAA' \subset \Com^+(\AAA)$ isomorphic to $\bar{\mu}|_{\AAA'}$ are precisely given by $\bar{\mu}_{\eta} = \bar{\mu}|_{\AAA'} + d(\eta)\epsilon$ for some $\eta \in \CC^1(\AAA')$. The existence of an $\eta$ for which $(\bar{\mu}_{\eta})_0 = 0$ (and hence for which the deformation is dg) is equivalent to the existence of $\delta' \in \prod_{C \in \AAA'} \AAA'(C,C)^1$ with $(d(\delta'))_0 = \phi(\delta, \delta)$, in other words to the fact that
$$0 = \chi_{\AAA}(\phi)_C \in H^2\Com^+(\AAA)(C,C)$$
for every $C \in \AAA'$. Clearly, $\AAA' = \Com^+(\AAA)_{\phi-\infty}$ is maximal with this property.

Finally, the statement concerning $\rho$ easily follows from the observation that for every $[\phi] \in HH^2(\AAA)$, $\AAA \subset \Com^+(\AAA)_{\phi-\infty}$.
\end{proof}

\subsection{Deformations of derived categories}\label{parderdef}
Let $\aaa$ be an abelian category with enough injectives.
Putting $\AAA = \Inj(\aaa)$ in the previous section, we obtain a bijection 
$$\Sk(R): HH^2_{\mathrm{ab}}(\aaa) \lra \Sk(\mathrm{cdg}-\Def_{D_{\mathrm{dg}}^+(\aaa)}(k[\epsilon]))$$
and the morphism $\rho'$ translates into
$$\rho': HH^2_{\mathrm{ab}}(\aaa) \lra  \Sk(\mathrm{dg}-\pDef_{D_{\mathrm{dg}}^+(\aaa)}(k[\epsilon]))$$
The following now immediately follows from Theorem \ref{precomplexes}:

\begin{theorem}\label{abelian}
Consider $\phi \in Z^2\CC_{\mathrm{ab}}(\aaa)$ and the corresponding abelian deformation $\aaa_{\phi}$ of $\aaa$.
Consider the subcategory $D_{\mathrm{dg}}^+(\aaa)_{\phi-\infty} \subset D_{\mathrm{dg}}^+(\aaa)$
spanned by the complexes $C$ with
$$0 = \chi_{\aaa}(\phi)_C \in D^+(\aaa)(C,C[2])$$
For every collection $\Gamma = \{\bar{C}\}_{C \in D_{\mathrm{dg}}^+(\aaa)_{\phi-\infty}}$ of bounded below complexes of $\aaa_{\phi}$-injectives with $k \otimes_{k[\epsilon]} \bar{C} = C$, the subcategory $D^+_{\mathrm{dg},\Gamma}(\aaa_{\phi}) \subset D_{\mathrm{dg}}^+(\aaa_{\phi})$ spanned by $\Gamma$ is a dg deformation of $D_{\mathrm{dg}}^+(\aaa)_{\phi-\infty}$ and a maximal partial dg deformation of $D^+_{\mathrm{dg}}(\aaa)$ representing $\rho'([\phi])$.
Consequently, $\rho'$ factors over an injection 
$$\rho: HH^2_{\mathrm{ab}}(\aaa) \lra \Sk(\mathrm{dg}-\mathrm{mpDef}_{D^+_{\mathrm{dg}}(\aaa_{\phi})}(k[\epsilon]))$$
The image consists of those maximal partial  dg deformations that are dg deformations of some $\AAA$ with $\Inj(\aaa) \subset \AAA \subset D^+_{\mathrm{dg}}(\aaa)$.
\end{theorem}

\begin{remark}\label{abelianremark}
Theorem \ref{abelian} suggests that we may consider $D_{\mathrm{dg}}^+(\aaa_{\phi})$ as a representative of the class of maximal partial dg deformations of $D_{\mathrm{dg}}^+(\aaa)$ corresponding to $[\phi]$ (see also Remark \ref{remcdgdefo}).
\end{remark}

\begin{remark}\label{boundedremark}
A ``bounded'' version of Theorem \ref{abelian} also holds true: simply replace every dg category $D^+_{\mathrm{dg}}$ in the theorem by its bounded version $D^b_{\mathrm{dg}} \subset D^+_{\mathrm{dg}}$ spanned by the complexes with bounded cohomology (see also \cite[\S 6.3]{lowen2}).
\end{remark}

As the maps $\Sk(R)$ and $\rho$ are not entirely satisfactory, we propose another sense in which to deform (exact) dg categories, that seems more adapted to (models of) derived categories of abelian categories.

For a commutative ring $R$, let $\dgcat(R)$ denote the (large) category of $R$-linear dg categories. In \cite{tabuada}, Tabuada defined a model structure on $\dgcat(R)$ for which the weak equivalences are the quasi-equivalences of dg categories. Let $\hodgcat(R)$ denote the homotopy category for this model structure. In \cite{toen}, To{\"e}n showed that  $\hodgcat(R)$ is a closed tensor category, with the derived tensor product $\otimes_R^L$ of dg categories, and with an internal hom between dg categories $\AAA$ and $\BBB$, which we will denote $\mathcal{RH}om_R(\AAA, \BBB)$, but which is \emph{not} a derived version of the internal hom of $\dgcat(R)$ for the above model structure (in fact it does have a derived interpretation for another model structure defined in \cite{tabuada2}). The forgetful functor $\hodgcat(k) \lra \hodgcat(k[\epsilon])$ has the left adjoint
$$k \otimes^L_{k[\epsilon]} -: \hodgcat(k[\epsilon]) \lra \hodgcat(k)$$
and the right adjoint
$$\mathcal{RH}om_{k[\epsilon]}(k,-): \hodgcat(k[\epsilon]) \lra \hodgcat(k)$$

\begin{definition}\label{exdef}
Let $\AAA$ be a $k$-linear dg category.
\begin{enumerate}
\item A first order \emph{homotopy dg deformation} of $\AAA$ is a $k[\epsilon]$-linear dg category $\BBB$ together with an isomorphism $k \otimes^L_{k[\epsilon]} \BBB \cong \AAA$ in $\hodgcat(k)$.
\item If $\AAA$ is exact, a first order \emph{exact homotopy dg deformation} of $\AAA$ is a $k[\epsilon]$-linear exact dg category $\BBB$ together with an isomorphism $\AAA \cong \mathcal{RH}om_{k[\epsilon]}(k,\BBB)$ in $\hodgcat(k)$.
\end{enumerate}
\end{definition}
Using the techniques of \cite{toen}, it is not hard to show the following

\begin{proposition}
Let $\aaa$ be an abelian $k$-linear category and suppose $\bbb$ is a flat abelian deformation of $\aaa$. Then $D_{\mathrm{dg}}(\bbb)$ is an exact homotopy dg deformation of $D_{\mathrm{dg}}(\aaa)$.
\end{proposition}

The further investigation of Definition \ref{exdef} (and its variations with respect to other model structures on dg categories (\cite{tabuada3, tabuada2, tabuada}) is part of a work in progress.

\begin{acknowledgements}
The author is grateful to Bernhard Keller for his continuous interest in, and many pleasant discussions on the topic of this paper.
\end{acknowledgements}

\def\cprime{$'$}
\providecommand{\bysame}{\leavevmode\hbox to3em{\hrulefill}\thinspace}
\bibliography{Bibfile}
\bibliographystyle{amsplain}

\end{document}